\documentclass[12pt,reqno]{amsart}
\usepackage{amsfonts}
\usepackage{a4wide}
\usepackage{amssymb}
\usepackage{amsmath}
\usepackage{amscd}
\usepackage{amssymb}
\usepackage{amsmath}
\usepackage{graphicx}
\usepackage{array}
\usepackage{bm}
\allowdisplaybreaks
\numberwithin{equation}{section}

\newtheorem{theorem}{Theorem}[section]
\newtheorem{proposition}[theorem]{Proposition}
\newtheorem{corollary}[theorem]{Corollary}
\newtheorem{lemma}[theorem]{Lemma}

\theoremstyle{definition}

\newtheorem{remark}[theorem]{Remark}

\newcommand{\R}{\mathbb{R}}

\begin{document}

\title
 [Kirchhoff type equations with general critical growth]
 {Multiple positive solutions for a class of Kirchhoff type problems involving general critical growth}

\begin{center}
\author{Liejun Shen}

\author{Xiaohua Yao}
\end{center}
\footnotetext{Email addresses: 	
liejunshen@sina.com (L. Shen), yaoxiaohua@mail.ccnu.edu.cn (X. Yao)}

\begin{center}
\address{(L. Shen and X. Yao) Hubei Key Laboratory of Mathematical Sciences and School of Mathematics and Statistics, Central China
Normal University, Wuhan, 430079, P. R. China }
\end{center}
\maketitle
\begin{abstract} In this paper, we study the following nonlinear Kirchhoff problems involving critical growth:
$$
\left\{%
\begin{array}{ll}
   -(a+b\int_{\Omega}|\nabla u|^2dx)\Delta u=|u|^4u+\lambda|u|^{q-2}u, \\
   u=0\ \ \text{on}\ \ \partial\Omega, \\
\end{array}%
\right.
$$
where $1<q<2$, $\lambda,\ a,\ b>0$ are parameters and $\Omega$ is a bounded domain in $\R^3$. We prove that there exists $\lambda_1=\lambda_1(q,\Omega)>0$ such that for any $\lambda\in(0,\lambda_1)$ and  $a,\ b>0$, the above Kirchhoff problem possesses at least two positive solutions and one of them is a positive ground state solution. We also establish the convergence property of the ground state solution as the parameter $b\searrow 0$.
More generally, we obtain the same results about the following Kirchhoff problem:
$$
\left\{%
\begin{array}{ll}
   -(a+b\int_{\mathbb{R}^3}|\nabla u|^2dx)\Delta u+u=Q(x)|u|^4u+{\lambda}f(x)|u|^{q-2}u, \\
   u\in H^1(\mathbb{R}^3), \\
\end{array}%
\right.
$$
 for any $a,\ b>0$ and $\lambda\in \big(0,\lambda_0(q,Q,f)\big)$ under certain conditions of $f(x)$ and $Q(x)$. Finally,
 we investigate the depending relationship between $\lambda_0$ and $b$ to
 show that for any (large) $\lambda>0$, there exists a $b_0(\lambda)>0$ such that the above results hold when $b>b_0(\lambda)$ and $a>0$.
\\

\end{abstract}

\section{Introduction and Main Results}

\setcounter{equation}{0}
In recent years, the following Kirchhoff type problem
\begin{equation}\label{a}
\left\{ \begin{gathered}
   - \Bigl( a+b\int_{\Omega}|\nabla u|^2dx \Bigr)\Delta u= f(x,u)\ \ \text{ in }\ \ \Omega, \hfill \\
 u=0\ \ \text{on}\ \ \partial\Omega, \hfill \\
\end{gathered}  \right.
\end{equation}
has been studied extensively by many researchers, here $f\in C(\Omega\times\R,\R)$, $\Omega\subset \R^N$, $N\ge 1$ and $a,b>0$ are constants.
Problems like \eqref{a} are seen to be nonlocal because of the appearance of the term $b(\int_{\Omega}|\nabla u|^2dx)\Delta u$ which implies that \eqref{a} is not a pointwise identity any more.  It is degenerate if $b=0$ and non-degenerate otherwise. The non-degenerate case causes some mathematical difficulties, which make the study of \eqref{a} interesting. The above problem \eqref{a} is related to the stationary analogue of the Kirchhoff equation
\begin{equation}\label{b}
 u_{tt}-\Bigl(a+b\int_{\Omega}|\nabla u|^2dx\Bigr)\Delta u=f(x,u)
\end{equation}
which was proposed by Kirchhoff in \cite{ap} as an extension of the classical D'Alembert's wave equation for free vibrations of elastic strings. For some mathematical and physical background on Kirchhoff type problems, we refer the readers to \cite{apa,aq,ar,as} and the references therein.
After J. L. Lions in his pioneer work \cite{bc} introduced an abstract functional analysis framework to \eqref{b}, the equation \eqref{b} has received an increasing attention on mathematical research. Recently, many important results about the the nonexistence, existence, and multiplicity of solutions for problem \eqref{a} have been obtained
with the nonlinear term $f(x,u)$ behaves like $|u|^{p-2}u$, here $2<p\leq2^*$, $2^*=\frac{2N}{N-2}$ if $N\geq 3$, $2^*=\infty$ if $N=1,2$. Please see for example
\cite{bl1,bl2,bj,bw,c} and the references therein.

In their celebrated paper, A. Ambrosetti, H. Br\'{e}zis, G. Cerami \cite{ckw} studied the following semilinear elliptic equation with concave-convex nonlinearities:
\begin{equation}\label{c}
\left\{ \begin{gathered}
   - \Delta u = |u|^{p-2}u+\lambda |u|^{q-2}u{\text{ in }}\Omega , \hfill \\
  u = 0{\text{ on }}\partial \Omega, \hfill \\
\end{gathered}  \right.
\end{equation}
where $\Omega$ is a bounded domain in $\R^N$ with $\lambda>0$ and $1<q<2<p\leq 2^*=\frac{2N}{N-2}$. By variational method, they have obtained the existence and multiplicity of positive solutions to the problem \eqref{c}. Subsequently, a series of similar results including the more general nonlinearity like $g(x)|u|^{p-2}u+\lambda f(x)|u|^{q-2}u$ are established, e.g. see \cite{cl1,cl2,cl3,cl4,cl5} and their reference therein. Some other types of multiplicity of positive solutions to the Schordinger or Schordinger-Poisson equations are established, see \cite{cl6,cl7,cl8,cl9} for example.

Using Nehair manifold and fibering map, C. Y. Chen, Y. C. Kuo, T. F. Wu \cite{h} extend the analysis to the Kirchhoff type equation
\begin{equation}\label{d}
\left\{ \begin{gathered}
   - M\Bigl(\int_{\Omega}|\nabla u|^2dx \Bigr)\Delta u =g(x)|u|^{p-2}u+\lambda h(x)|u|^{q-2}u{\text{ in }}{\Omega}, \hfill \\
 u = 0{\text{ on }}\partial \Omega, \hfill \\
\end{gathered}  \right.
\end{equation}
where $\Omega$ is a smooth bounded domain in $\R^N$ with $1<q<2<p<2^*=\frac{2N}{N-2}$. $M(s)=as+b$, the parameters $a,b,\lambda>0$ and the weight functions $h,g\in C(\overline{\Omega})$ satisfy some specified conditions,
they show that there exist at least two positive solutions when $0<\lambda<\lambda_0(a)$ and $4<p<6$ for the problem \eqref{d}, here $\lambda_0(a)$ strongly relies on $a>0$.
In addition, the existence
and multiplicity of solutions for Kirchhoff type problems in the whole space
$\R^N$ has been established in \cite{hb,hc,hc1} are taken into account.

In the critical growth case (i.e. $p=\frac{2N}{N-2}$), very recently, C. Y. Lei, G. S. Liu, L. T. Guo in \cite{i} consider the following Kirchhoff problem in three dimensions:
\begin{equation}\label{dd}
\left\{ \begin{gathered}
   - \Bigl(a+\epsilon\int_{\Omega}|\nabla u|^2dx \Bigr)\Delta u = u^5+\lambda u^{q-1}{\text{ in }}{\Omega}, \hfill \\
  u = 0{\text{ on }}\partial \Omega,\hfill \\
\end{gathered}  \right.
\end{equation}
where $\Omega$ is a smooth bounded domain in $\R^3$, $a>0$, $1<q<2$, $\epsilon>0$ is small enough, $\lambda>0$ is a positive real number. They obtained that there exists a $\lambda_\star>0$ such that when $\epsilon>0$ is small enough and $\lambda\in (0,\lambda_\star)$, the problem \eqref{dd} has at least two positive solutions.

Inspired by the works mentioned above, particularly, by the results in \cite{i}, we try to get the existence of positive solutions of the Kirchhoff problem for the critical growth. To be precise, we study the problem
\begin{equation}\label{f}
\left\{ \begin{gathered}
   - \Bigl(a+b\int_{\Omega}|\nabla u|^2dx \Bigr)\Delta u= |u|^4u+\lambda|u|^{q-2}u,{\text{ in }}{\Omega}, \hfill \\
 u = 0{\text{ on }}\partial \Omega. \hfill \\
\end{gathered}  \right.
\end{equation}
\vskip0.3cm
\begin{theorem}\label{corollary}
Let $1<q<2$. Then  there exists a $\lambda_1=\lambda_1(q,\Omega)>0$ such that
for any $\lambda\in(0,\lambda_1)$ and any $a,b>0$, the problem (\ref{f}) has at least two positive solutions in $H^1_0(\Omega)$, and one of the solutions is a positive ground state solution. In particular, let $\lambda\in(0,\lambda_1)$ be fixed, for any sequence $\{b_n\}$ with $b_n\searrow 0$ as $n\to \infty$, there exists a subsequence (still denoted by $\{b_n\}$) such that $u_{b_n}$ converges to $w_0$ strongly in $H^1_0(\Omega)$ as $n\to \infty$, where $u_{b_n}$ is a positive ground state solution of the problem \eqref{f} and $w_0$ is a solution of the problem
$$
\left\{
  \begin{array}{ll}
    -a\Delta u=|u|^4u+\lambda|u|^{q-2}u,\ \ \text{in}\ \ \Omega,  \\
    u=0\ \ \text{on}\ \ \partial\Omega.
  \end{array}
\right.
$$

\end{theorem}
\vskip0.3cm
\begin{remark}\label{1.1.}
 Comparing with \cite{i}, we only assume $b>0$ is a positive constant not necessary to be small enough. Therefore, we greatly relax the constraints on the parameter $\epsilon$ in \cite{i}. Moreover, we  obtain the above convergence property of ground solution of the problem (\ref{f}) as $b\searrow 0$.
\end{remark}
\vskip0.3cm
In this paper we will not give the direct proof of Theorem \ref{corollary}, which can be seen as a corollary of the following general problem
\begin{equation}\label{g}
\left\{ \begin{gathered}
   - \Bigl(a+b\int_{\R^3}|\nabla u|^2dx \Bigr)\Delta u + u = Q(x)|u|^4u+\lambda f(x)|u|^{q-2}u,{\text{ in }}{\mathbb{R}^3}\hfill \\
 {\text{ }}u \in {H^1}({\mathbb{R}^3})\hfill \\
\end{gathered}  \right.
\end{equation}
under some assumptions on the weight functions $Q(x)$ and $f(x)$. Actually, we stress here that the method dealing with problem \eqref{g} can be applied directly in the problem \eqref{f} just by simply modifying some notations such as $Q(x)=f(x)\equiv 1$, and $\Omega$ instead of $\R^3$ in problem \eqref{f}.

Now the main results for \eqref{g} can be stated as follows:
\vskip0.3cm
\begin{theorem}\label{main theorem}
Assume $1<q<2$, and the functions $Q,f$ satisfy the following conditions:\\
$(i)$ $0<f(x)\in L^{\frac{6}{6-q}}(\R^3)$;\ $(ii)$ $Q(x)\geq 0$, and $\exists x_0\in \mathbb{R}^3$, $\alpha\in(\frac{q}{2},+\infty)$, $\rho>0$ such that $$Q(x_0)=\max\limits_{x\in \mathbb{R}^3}Q(x):=|Q|_{\infty}<+\infty$$ and
$$
|Q(x)-Q(x_0)|\leq C|x-x_0|^\alpha\ \ \text{whenever} \ \ |x-x_0|<\rho,
$$
for some constant $C>0$. Then there exists  a $\lambda_0=\lambda_0(q,Q,f)>0$ such that
for any $\lambda\in(0,\lambda_0)$ and any $a,b>0$, the problem \eqref{g} has at least two positive solutions in $H^1(\R^3)$, and one of the solutions is a positive ground state solution.
\end{theorem}
\vskip0.3cm
\begin{remark}\label{remark1}\
$(1)$ The condition $(i)$ in Theorem \ref{main theorem} that \emph{$f(x)$ is positive} is just for simplicity. In fact, our method can deal with the case that \emph{$f(x)$ is sign-changing}.

 $(2)$ It is clear that $Q(x)\equiv 1$ satisfies the condition $(ii)$ in Theorem \ref{main theorem}. Also, $Q(x_0)>0$ is necessary  and otherwise $Q(x)\equiv 0$ implies that the critical term disappears. The general critical term $Q(x)|u|^4u$ was firstly introduced in \cite{jj} by F. Gazzola and M. Lazzarino,
 then this condition has already been extended to Sch\"{o}rdinger-Possion system in \cite{jj2}
and Kirchhoff type problem in \cite{jj3} with $\alpha\in[1,3)$. Compared with \cite{jj2}
and \cite{jj3}, the condition $(ii)$ has a bit improvement, since we just assume that $\alpha\in(\frac{q}{2},+\infty)$ for
$q\in(1,2)$.
\end{remark}
\vskip0.3cm
By the results in Theorem 1.3, we know that the constant $\lambda_0$ is independent of $b>0$.  So we are interested to know what happens if $b>0$ is sufficiently small in our problem \eqref{g}. Here we can give the following theorem:
\begin{theorem}\label{theorem2} Under the assumptions of Theorem \ref{main theorem} and $\lambda\in(0,\lambda_0)$ is fixed, then for any sequence $\{b_n\}$ with $b_n\searrow 0$ as $n\to \infty$, there exists a subsequence (still denoted by $\{b_n\}$) such that $u_{b_n}$ convergent to $w_0$ strongly in $H^1(\R^3)$ as $n\to \infty$, where $u_{b_n}$ is a positive ground state solution to problem \eqref{g} and $w_0$ is a solution of the problem
\begin{equation}\label{h}
\left\{ \begin{gathered}
   - a\Delta u + u = Q(x)|u|^4u+\lambda f(x)|u|^{q-2}u,{\text{ in }}{\mathbb{R}^3}, \hfill \\
 {\text{ }}u \in {H^1}({\mathbb{R}^3}). \hfill \\
\end{gathered}  \right.
\end{equation}
\end{theorem}
\vskip0.3cm
For $\lambda_0$ (independent on $b$) in Theorem \ref{main theorem} and Theorem \ref{theorem2}, it is unknown to us whether the constant $\lambda_0$ is a best one or not for all $b>0$.  However, for a fixed large $b>0$,  it is far from the best constant. Indeed, in the following theorem we have proved that there exists $\tilde{\lambda}_0>\lambda_0$  such that results above holds true for any $\lambda\in(0,\tilde{\lambda}_0)$ and any $a>0$.
\vskip0.3cm
\begin{theorem}\label{theorem3}
Let $q,Q,f$ satisfy the assumptions as the Theorem \ref{main theorem}, then for any $b>0$, there exists  $\tilde{\lambda}_0=Cb^{\frac{6-q}{2}}$, here $C=C(q,Q,f)>0$ is independent on $b$, such that
for any $\lambda\in(0,\tilde{\lambda}_0)$ and any $a>0$, the problem \eqref{g} has at least two positive solutions in $H^1(\R^3)$, and one of the solutions is a positive ground state solution.
\end{theorem}
Obviously, $\tilde{\lambda}_0=Cb^{\frac{6-q}{2}}$ converges to $\infty$ as $b\nearrow \infty$.   Hence, combining with Theorem \ref{main theorem} and Theorem \ref{theorem3}, we have the following conclusion:
 \begin{corollary}\label{theorem4}
Let $q,Q,f$ satisfy the assumptions as the Theorem \ref{main theorem}, then for any $\lambda\in(0,+\infty)$ there exists $b_0(\lambda)\ge 0$ such that problem \eqref{g} has at least two positive solutions in $H^1(\R^3)$, and one of the solutions is a positive ground state solution when $b>b_0(\lambda)$ and any $a>0$.
\end{corollary}
\begin{remark}\label{remark2}
To the best knowledge of us, the Corollary \ref{theorem4} seems new for the Kirchhoff problem involving critical growth in the whole space $\R^3$. Considered if $b=0$, it seems failure for the degenerate case for Kirchhoff problem.
\end{remark}
\begin{remark}\label{remark3}
Corollary \ref{theorem4} is a direct conclusion of the Theorem \ref{main theorem} and Theorem \ref{theorem3}, and we can give in fact the explicit definition of $b_0(\lambda)$ as the following:
$$
 b_0(\lambda)=\left\{
 \begin{array}{ll}
    0, & \text{if}\ \ \lambda<\lambda_0, \\
    \tilde{C}\lambda^{\frac{2}{6-q}}, & \text{if}\ \ \lambda\geq{\lambda}_0,
  \end{array}
\right.
$$
where $\lambda_0=\lambda_0(q,Q,f)$ and $\tilde{C}=\tilde{C}(q,Q,f)>0$ are independent on $b$. On the other hand,
We mention here that Corollary \ref{theorem4} remains true for problem \eqref{f} on bounded domain $\Omega$.
\end{remark}
\vskip0.3cm
Before we turn to next section,  we would like  to mention some main ideas of the proof of Theorems 1.3, 1.5 and 1.6. It seems that the methods used in \cite{h} and \cite{i} can not be applied directly in our paper. On one hand, in \cite{h}, the method used the Nehair manifold and fibering map will be not directly applied because the norm can't be defined as $(\int_{\R^3}|\nabla u|^2dx)^{\frac{1}{2}}$. On the other hand, the first solution in \cite{i} is easy to be found, and however, the procedure by which the authors of \cite{i}  used the first solution to find the second one,  seems strongly to depend on $b>0$ small enough.

\vskip0.2cm
To establish Theorem 1.3, we will used the Mountain-pass theorem and the Ekeland's variational principle \cite{j} which indicate that the $(PS)$ condition of the energy functional $I$ is necessary.
However, the functional $I$ does not satisfy the $(PS)$ condition at every energy level $c$ because of the appearance of critical term. To overcome this difficulty, we try to pull the energy level down below some critical level.
It is more complicated to
handle the nonlocal effect which does not imply that $\int_{\R^3}|\nabla u_n|^2dx\to\int_{\R^3}|\nabla u|^2dx$ from $u_n \rightharpoonup u$ in $H^1(\R^3)$. Therefore we estimate a critical value
$$
c<\frac{abS^3}{4|Q|_{\infty}}+\frac{b^3S^6}{24|Q|^2_{\infty}}+\frac{(b^2S^4+4a|Q|_{\infty}S)^{\frac{3}{2}}}{24|Q|^2_{\infty}}-C_i\lambda^{\frac{2}{2-q}},
\ \  C_1,C_2\in (0,+\infty).
$$
different from J. Wang et al \cite{jj4} and G. Li, H. Ye \cite{jj5} to recover the compactness condition.
Then we can prove Theorem 1.3.
\vskip0.2cm
To show Theorem 1.5, we follow the idea used in W. Shuai \cite{jj5a} and
X. Tang, B. Cheng \cite{jj6} to consider what happens when $b\searrow 0$.
Fixing $b\in(0,1]$ and then we get a bounded sequence of positive ground state solutions to \eqref{g}. As a consequence of that the functional $I(u)$ satisfies the $(PS)$ condition at some level, the proof of Theorem 1.5 is clear.
\vskip0.2cm
In Theorem 1.6, we concern how  $\lambda_0$ can be determined by $b$. To do it, we used a different approach (see Lemma \ref{5.1}) to find a $(PS)$ sequence. Considering the effect from $b$,
 the proof of $(PS)$ condition of the energy functional $I(u)$ may be different from Lemma \ref{2.3}. But the weak limit $u\in H^1(\R^3)$ of
a $(PS)_{\tilde{c}}$ sequence $\{u_n\}$ is a critical point of the following functional:
$$
J(u):=\frac{a+bA^2}{2}\int_{\R^3}|\nabla u|^2dx+\frac{1}{2}\int_{\R^3}|u|^2dx-\frac{1}{6}\int_{\R^3}Q(x)|u|^6dx-\frac{\lambda}{q}\int_{\R^3}f(x)|u|^q,
$$
and the $\{u_n\}$ is a $(PS)_{\tilde{c}+\frac{bA^2}{4}}$
sequence for $J(u)$, where $A^2=\lim_{n\to\infty}\int_{\R^3}|\nabla u_n|^2dx$. We try to prove that $I(u)$ satisfies $(PS)_c$ condition with the help
of $J(u)$ (see Lemma \ref{5.3}). Hence, the proof Theorem 1.6 is complete.
\vskip0.3cm
This paper is organized as follows: In Section 2, we give some notations and crucial lemmas. In
Section 3, we prove the existence of the two different positive solutions and the positive ground state solution of problem \eqref{g}. In Section 4, we analyze the convergence property of the positive ground state solution of problem \eqref{g} and prove the Theorem \ref{theorem2}. In Section 5, we establish  the proof of Theorem \ref{theorem3}.

\section{ Some Notations and Lemmas}

\setcounter{equation}{0}
In this section, we first give several notations and definitions. Throughout this paper, $L^p(\R^3)$ $(1\leq p\leq \infty)$ is the usual Lebesgue space with the standard norm $|u|_p$. We use $``\to"$ and $``\rightharpoonup"$ to denote the strong and weak convergence in the related function space, respectively. For any $\rho>0$ and any $x\in \R^3$, $B_\rho(x)$ denotes the ball of radius $\rho$ centered at $x$, that is $B_\rho(x):=\{y\in \R^3:|y-x|<\rho\}$. $C$ will denote a
positive constant unless specified.

 We denote by $H^1(\R^3)$ the usual Sobolev space equipped with the norm
$$
\|u\|=\bigg(\int_{\R^3}a|\nabla u|^2+|u|^2dx\bigg)^{\frac{1}{2}}, \ \  \forall u\in H^1(\R^3),
$$
and $(H^1(\R^3))^*$ is the dual space of $H^1(\R^3)$.
Let $D^{1,2}(\R^3)$ be the completion of $C_0^{\infty}(\R^3)$ with respect to the Dirichlet norm,
$$
\|u\|_{D^{1,2}(\R^3)}:=\bigg(\int_{\R^3}|\nabla u|^2dx\bigg)^{\frac{1}{2}}, \ \  \forall u\in D^{1,2}(\R^3)
$$
and $S$ denote the best Sobolev constant, namely
\begin{equation}\label{2.a}
  S:=\inf_{u\in D^{1,2}(\R^3)\backslash \{0\}}\frac{\int_{\R^3}|\nabla u|^2dx}{(\int_{\R^3}|u|^6dx)^{\frac{1}{3}}},
\end{equation}
Define the energy functional $I:H^1(\R^3)\to\R$ of \eqref{g} by
$$
\begin{gathered}
I(u)=\frac{1}{2}\int_{\R^3}a|\nabla u|^2+|u|^2dx+\frac{b}{4}\bigg(\int_{\R^3}|\nabla u|^2dx\bigg)^2\hfill\\
-\frac{1}{6}\int_{\R^3}Q(x)|u|^6dx-\frac{\lambda}{q}\int_{\R^3}f(x)|u|^qdx,
\end{gathered}
$$
the functional $I$ is well-defined on $H^1(\R^3)$ and $I\in C^1(H^1(\R^3),\R)$ (see \cite{jj7}). What's more, for any $u,v\in H^1(\R^3)$, we have
$$
\begin{gathered}
\langle I^{\prime}(u),v\rangle=\int_{\R^3}a\nabla u\nabla v+uvdx+b\int_{\R^3}|\nabla u|^2dx\int_{\R^3}\nabla u\nabla v\hfill\\
-\int_{\R^3}Q(x)|u|^4uvdx-\lambda\int_{\R^3}f(x)|u|^{q-2}uvdx.
\end{gathered}
$$
Clearly, the critical points are the weak solutions of problem \eqref{g}. Furthermore, a sequence $\{u_n\}\subset H^1(\R^3)$ is a $(PS)$ sequence of the functional $I(u)$ at the level $d\in\R$ if $I(u_n)\to d$ and $I^{\prime}(u_n)\to 0$ as $n\to\infty$. And we say a $(PS)_d$ sequence satisfies the $(PS)_d$ condition if it contains a strong convergent subsequence.

The following Lemmas play vital roles in proving Theorem \ref{main theorem}:
\begin{lemma}\label{2.1}
For any $\lambda\in (0, \lambda_0)$, the functional $I(u)$ satisfies the Mount-pass geometry around $0\in H^1(\R^3)$, that is,

$(i)$ there exist $\eta,\beta>0$ such that $I(u)\geq \eta>0$ when $\|u\|=\beta$;

$(ii)$ there exists $e\in H^1(\R^3)$ with $\|e\|>\beta$ such that $I(e)<0$.
\end{lemma}
\begin{proof}
 $(i)$ By the definition of $I(u)$,
\begin{align*}
 I(u)&\geq \frac{1}{2}\|u\|^2-\frac{|Q|_\infty}{6}S^{-3}\|u\|^6-\frac{\lambda}{p}|f|_{\frac{6}{6-q}}\|u\|^q         \\
         & =\|u\|^q\bigg(\frac{1}{2}\|u\|^{2-q}-\frac{|Q|_\infty}{6}S^{-3}\|u\|^{6-q}-\frac{\lambda}{q}|f|_{\frac{6}{6-q}}\bigg)\\
&\geq \|u\|^q\bigg(C_0-\frac{\lambda}{q}|f|_{\frac{6}{6-q}}\bigg),
  \end{align*}
where $C_0=\frac{2}{6-q}\big[\frac{3S^3(2-q)}{|Q|_\infty(6-q)}\big]^{\frac{2-q}{4}}$. Therefore letting $\lambda_0=\frac{qC_0}{|f|_{\frac{6}{6-q}}}$, and there exist $\eta,\beta>0$ such that $I(u)\geq \eta>0$ when $\|u\|=\beta$ for any $\lambda\in (0,\lambda_0)$.

$(ii)$ It's clear that $\lim\limits_{t\to+\infty}I(tu_0)=-\infty$ for some $u_0\in H^1(\R^3)$, choosing $e=t_0u_0$ with $t_0$ large enough, we have $\|e\|>\beta$ and $I(e)<0$.
\end{proof}

By Lemma \ref{2.1}, we can find a $(PS)$ sequence of the functional $I(u)$ at the level
\begin{equation}\label{2.1a}
 c:=\inf_{\gamma\in \Gamma}\sup_{t\in [0,1]}I(\gamma(t)),
\end{equation}
where
\begin{equation}\label{2.1b}
  \Gamma:=\{\gamma\in C([0,1],H^1(\R^3)):\gamma(0)=0, I(\gamma(1)<0\}.
\end{equation}
Now a $(PS)_c$ sequence can be constructed, that is, there exists a sequence $\{u_n\}\subset H^1(\R^3)$ satisfies
\begin{equation}\label{2.1c}
  I(u_n)\to c,\ \ I^{\prime}(u_n)\to 0\ \  \text{as}\ \ n\to \infty.
\end{equation}

The functional $I(u)$ do not satisfy the $(PS)_c$ condition at every energy level because of the appearance of the critical term $Q(x)|u|^4u$, and in order to recover the compactness condition, we would have to estimate the critical energy carefully.
\vskip0.3cm
We know that $U(x)=\frac{3^{\frac{1}{4}}}{1+|x|^2}$ is a solution of
$$
-\Delta u=u^5, x\in \R^3,\eqno(\star)
$$
then for any $\epsilon>0$, $U_\epsilon(x)=\epsilon^{-\frac{1}{2}}U(\epsilon^{-1}x)$ is also a solution of $(\star)$ and satisfies that $\int_{\R^3}|\nabla U_\epsilon|^2dx=\int_{\R^3}|U_\epsilon|^6dx=S^{\frac{3}{2}}$. To estimate the critical energy, we chose a cut-off function $\varphi(x)$ satisfying $\varphi(x)\in C_0^{\infty}(B_R(x_0))$, $0\leq \varphi(x)\leq 1$ in $\R^3$ and $\varphi(x)\equiv 1$ on $B_{\frac{R}{2}}(x_0)$. Define
\begin{equation}\label{2.1d}
  v_\epsilon(x)=\varphi(x)U_\epsilon(x)=\frac{\varphi(x)(3\epsilon^2)^{\frac{1}{4}}}{(\epsilon^2+|x-x_0|^2)^{\frac{1}{2}}},
\end{equation}
thanks to the results in \cite{ja}, we have
\begin{equation}\label{2.1e}
  |\nabla v_\epsilon|_2^2=K_1+O(\epsilon),\ \  |v_\epsilon|_6^2=K_2+O(\epsilon)
\end{equation}
with $\frac{K_1}{K_2}=S$ and for any $s\in [2,6)$,
\begin{equation}\label{2.1f}
  |v_\epsilon|_s^s=
\left\{
  \begin{array}{ll}
   O(\epsilon^{\frac{s}{2}}), & \text{if}\ \ s\in[2,3), \\
     O(\epsilon^{\frac{3}{2}}|\log \epsilon|), & \text{if}\ \ s=3, \\
     O(\epsilon^{\frac{6-s}{2}}), &   \text{if}\ \ s\in(3,6).
  \end{array}
\right.
\end{equation}
\vskip0.3cm
\begin{lemma}\label{2.2}
Assume $1<q<2$ and $\lambda\in(0,\lambda_0)$, then the critical energy
$$
c<\frac{abS^3}{4|Q|_{\infty}}+\frac{b^3S^6}{24|Q|^2_{\infty}}+\frac{(b^2S^4+4a|Q|_{\infty}S)^{\frac{3}{2}}}{24|Q|^2_{\infty}}-C_1\lambda^{\frac{2}{2-q}},
$$
where $C_1=\frac{2-q}{4q}\bigg[\frac{(4-q)|f|_{\frac{6}{6-q}}}{2S^{\frac{q}{2}}}\bigg]^{\frac{2}{2-q}}$ is a positive constant, and $S$ is the best Sobolev constant given in \eqref{2.a}.
\end{lemma}
\vskip0.3cm
\begin{proof}
 Firstly, we claim that there exist $t_1,t_2\in(0,+\infty)$ independent of $\epsilon,\lambda$ such that $\max\limits_{t\geq 0}I(tv_\epsilon)=I(t_\epsilon v_\epsilon)$ and
\begin{equation}\label{2.1g}
  0<t_1<t_\epsilon<t_2<+\infty.
\end{equation}
Indeed, by the fact that $\lim\limits_{t\to +\infty}I(tv_\epsilon)=-\infty$ and $(i)$ of Lemma \ref{2.1}, there exists $t_\epsilon>0$ such that
$$
\max\limits_{t\geq 0}I(tv_\epsilon)=I(t_\epsilon v_\epsilon), \frac{d}{dt}I(tv_\epsilon)=0, \frac{d^2}{dt^2}I(tv_\epsilon)<0
$$
which imply that
\begin{equation}\label{2.1h}
  t_\epsilon\|v_\epsilon\|^2+t_\epsilon^3\bigg(\int_{\R^3}|\nabla v_\epsilon|^2dx\bigg)^2-t_\epsilon^6\int_{\R^3}Q(x)|v_\epsilon|^6dx-\lambda t_\epsilon^q\int_{\R^3}f(x)|v_\epsilon|^qdx=0
\end{equation}
and
\begin{equation}\label{2.1i}
  \|v_\epsilon\|^2+3t_\epsilon^2\bigg(\int_{\R^3}|\nabla v_\epsilon|^2dx\bigg)^2-6t_\epsilon^5\int_{\R^3}Q(x)|v_\epsilon|^6dx-q\lambda t_\epsilon^{q-1}\int_{\R^3}f(x)|v_\epsilon|^qdx<0.
\end{equation}
It follows from \eqref{2.1h} that $t_\epsilon$ is bounded from above and \eqref{2.1i} that $t_\epsilon$ is bounded from below, then \eqref{2.1g} is true.

The energy level $c$ given in \eqref{2.1a} tells us it's enough to show for any sufficient small $\epsilon>0$, $\max\limits_{t\geq 0}I(tv_\epsilon)<\frac{abS^3}{4|Q|_{\infty}}+\frac{b^3S^6}{24|Q|^2_{\infty}}+\frac{(b^2S^4+4a|Q|_{\infty}S)^{\frac{3}{2}}}{24|Q|^2_{\infty}}-C_1\lambda^{\frac{2}{2-q}}$ holds. In fact, for the given $v_\epsilon$ in \eqref{2.1d}, there exists a large $T>0$ such that $I(Tv_\epsilon)<0$. Let $\widetilde{\gamma}(t)=tTv_\epsilon\in \Gamma$, where $\Gamma$ is defined as in \eqref{2.1b}, then $c\leq \max_{0\leq t\leq 1}I(\widetilde{\gamma}(t))=\max_{0\leq t\leq 1}I(tTv_\epsilon)\leq \max_{t\geq 0}I(tv_\epsilon)$. It follows from the definition of $I(u)$ that
\begin{equation}\label{2.1j}
  \begin{gathered}
I(tv_\epsilon)=\frac{t^2}{2}\int_{\R^3}a|v_\epsilon|^2+|v_\epsilon|^2dx+\frac{bt^4}{4}\bigg(\int_{\R^3}|\nabla v_\epsilon|^2dx\bigg)^2\hfill\\
-\frac{t^6}{6}\int_{\R^3}Q(x)|v_\epsilon|^6dx-\frac{\lambda t^q}{q}\int_{\R^3}f(x)|v_\epsilon|^qdx,
\end{gathered}
\end{equation}
and then set
\begin{align*}
g(t)&:=\frac{t^2}{2}\|v_\epsilon\|^2+\frac{bt^4}{4}\bigg(\int_{\R^3}|\nabla v_\epsilon|^2dx\bigg)^2-\frac{t^6}{6}\int_{\R^3}Q(x_0)|v_\epsilon|^6dx\\
&=\widetilde{C_1}t^2+\widetilde{C_2}t^4-\widetilde{C_3}t^6,
\end{align*}
where
$$
\widetilde{C_1}=\frac{1}{2}\|u\|^2,\ \ \widetilde{C_2}=\frac{b}{4}\bigg(\int_{\R^3}|\nabla v_\epsilon|^2dx\bigg)^2,\ \   \widetilde{C_3}=\frac{1}{6}\int_{\R^3}Q(x_0)|v_\epsilon|^6dx.
$$
Combing with \eqref{2.1d}-\eqref{2.1e} and some elementary computations we have
\begin{equation}\label{2.1k}
  {
  \begin{split}
    \max_{t\geq 0}g(t) & =\frac{9\widetilde{C_1}\widetilde{C_2}\widetilde{C_3}+2\widetilde{C_2}^3+2(\widetilde{C_2}^2+3\widetilde{C_1}\widetilde{C_3})^{\frac{3}{2}}}{27\widetilde{C_3}^2} \\
      & =\frac{abS^3}{4|Q|_{\infty}}+\frac{b^3S^6}{24|Q|^2_{\infty}}+\frac{(b^2S^4+4a|Q|_{\infty}S)^{\frac{3}{2}}}{24|Q|^2_{\infty}}+O(\epsilon).
  \end{split}
  }
\end{equation}
On the other hand, for $\epsilon>0$ with $\epsilon<\frac{R}{2}$ we have
\begin{equation}\label{2.1kk}
  {\begin{split}
     \lambda t_\epsilon^q\int_{\R^3}f|v_\epsilon|^qdx &=\lambda t_\epsilon^q\int_{B_R(x_0)}f|v_\epsilon|^qdx \\
       & \stackrel{\mathrm{\eqref{2.1g}}}{\geq } C\lambda\int_{B_{\frac{R}{2}}(x_0)}f(x)\frac{\epsilon^{\frac{q}{2}}}{(\epsilon^2+|x|^2)^{\frac{q}{2}}}dx\\
   &\geq \lambda(\frac{2}{R^2})^{\frac{q}{2}}\epsilon^{\frac{q}{2}}\int_{B_{\frac{R}{2}}(x_0)}f(x)dx:= C_2\epsilon^{\frac{q}{2}},
   \end{split}
  }
\end{equation}
where $C_2\in (0,+\infty)$ since $f(x)\in L^{\frac{6}{6-q}}(\R^3)$ and $f(x)\in L^1_{loc}(\R^3)$. It follows from some direct computations:
$$
\epsilon^3\int_{B_\epsilon(x_0)}\frac{|x-x_0|^{\alpha}}{(\epsilon^2+|x-x_0|^2)^3}dx\leq
\frac{1}{\epsilon^3}\int_{B_\epsilon(x_0)}|x-x_0|^{\alpha}dx\leq C\epsilon^\alpha
$$
and
$$
\begin{gathered}
\epsilon^3\int_{B_\rho(x_0)\backslash B_\epsilon(x_0)}\frac{|x-x_0|^{\alpha}}{(\epsilon^2+|x-x_0|^2)^3}dx
\leq\epsilon^3\int_{B_\rho(x_0)\backslash B_\epsilon(x_0)}|x-x_0|^{\alpha-6}dx\hfill\\
=C\epsilon^3\int_{\epsilon}^{\rho}r^{\alpha-4}dr
=\left\{
                                                  \begin{array}{ll}
                                                    C\epsilon^\alpha, & \text{if}\ \  \alpha<3, \\
                                                    C\epsilon^3|\ln\epsilon|, &\text{if}\ \   \alpha=3, \\
                                                    C\epsilon^3,  & \text{if}\ \ \alpha>3,
                                                  \end{array}
                                                \right.\hfill\\
\end{gathered}
$$
that the following conclusion
\begin{equation}\label{2.1m}
  \begin{gathered}
t_\epsilon^6\int_{\R^3}|Q(x)-Q(x_0)||v_\epsilon|^6dx=t_\epsilon^6\int_{B_\rho(x_0)}|Q(x)-Q(x_0)||v_\epsilon|^6dx\hfill\\
\stackrel{\mathrm{(2.7)}}{\leq} C\epsilon^3\int_{B_\rho(x_0)}\frac{|x-x_0|^{\alpha}}{(\epsilon^2+|x-x_0|^2)^3}dx
\leq\left\{
                                                  \begin{array}{ll}
                                                    C\epsilon^\alpha, & \text{if}\ \  \alpha<3, \\
                                                    C\epsilon^3|\ln\epsilon|, &\text{if}\ \   \alpha=3, \\
                                                    C\epsilon^3,  & \text{if}\ \ \alpha>3,
                                                  \end{array}
                                                \right.\hfill\\
\end{gathered}
\end{equation}
holds. It is obvious that if $\alpha>3$, then there exists $\epsilon_1>0$ such for any $\epsilon\in (0,\epsilon_1)$ we have
\begin{equation}\label{2.1n}
  \frac{O(\epsilon)}{\epsilon}+C\epsilon^2-C_2\epsilon^{q-1}<-C_1\lambda^{\frac{2}{2-q}};
\end{equation}
if $\alpha=3$, then there exists $\epsilon_2>0$ such for any $\epsilon\in (0,\epsilon_2)$ we have
\begin{equation}\label{2.1nn}
 \frac{O(\epsilon)}{\epsilon}+C\epsilon^2|\ln\epsilon|-C_2\epsilon^{q-1}<-C_1\lambda^{\frac{2}{2-q}};
\end{equation}
if $\alpha\in(\frac{q}{2},3)$, then there exists $\epsilon_3>0$ such for any $\epsilon\in (0,\epsilon_3)$ we have
\begin{equation}\label{2.1nnn}
  \frac{O(\epsilon)}{\epsilon}+C\epsilon^2|\ln\epsilon|-C_2\epsilon^{q-1}<-C_1\lambda^{\frac{2}{2-q}}.
\end{equation}
Therefore combing with \eqref{2.1j}-\eqref{2.1nnn}, for any $\alpha>\frac{q}{2}$, there exists $\epsilon_0=\min\{\epsilon_1,\epsilon_2,\epsilon_3,\frac{2}{R}\}$ such that $\forall \epsilon\in(0,\epsilon_0)$,
\begin{align*}
\max\limits_{t\geq 0}I(tv_\epsilon)&=I(t_\epsilon v_\epsilon)=g(t_\epsilon)+t_\epsilon^6\int_{\R^3}|Q(x)-Q(x_0)||v_\epsilon|^6dx-\lambda t_\epsilon^q\int_{\R^3}f(x)|v_\epsilon|^qdx\\
&\leq\max_{t\geq 0}g(t)-C_2\epsilon^q+\left\{
                          \begin{array}{ll}
                            C\epsilon^\alpha, & \text{if}\ \  \alpha<3, \\
                             C\epsilon^3|\ln\epsilon|, &\text{if}\ \   \alpha=3,\\
                            C\epsilon^3,  & \text{if}\ \ \alpha>3
                          \end{array}
                        \right.\\
&<\frac{abS^3}{4|Q|_{\infty}}+\frac{b^3S^6}{24|Q|^2_{\infty}}+\frac{(b^2S^4+4a|Q|_{\infty}S)^{\frac{3}{2}}}{24|Q|^2_{\infty}}-C_1\lambda^{\frac{2}{2-q}}.
\end{align*}
Now the proof is complete.
\end{proof}
\vskip0.3cm
\begin{lemma}\label{2.3}
The functional $I(u)$ satisfies the $(PS)_c$ condition when $\lambda\in (0,\lambda_0)$ and
$$
c<\frac{abS^3}{4|Q|_{\infty}}+\frac{b^3S^6}{24|Q|^2_{\infty}}+\frac{(b^2S^4+4a|Q|_{\infty}S)^{\frac{3}{2}}}{24|Q|^2_{\infty}}-C_1\lambda^{\frac{2}{2-q}}.
$$
\end{lemma}
\vskip0.3cm
\begin{proof}
When $\lambda\in (0,\lambda_0)$, by Lemma \ref{2.1}, there exists a sequence $\{u_n\}$ assuming that
$$
I(u_n)\to c,\ \ I^{\prime}(u_n)\to 0\ \  \text{as}\ \ n\to \infty.
$$
Then
\begin{align*}
 c+1+o(1)\|u_n\|&\geq I(u_n)-\frac{1}{4}\langle I^{\prime}(u_n),u_n\rangle \\
       &=\frac{1}{4}\|u_n\|^2+\frac{1}{12}\int_{\R^3}Q(x)|u_n|^6dx-\lambda(\frac{1}{q}-\frac{1}{4})\int_{\R^3}f|u_n|^qdx        \\
         & \geq \frac{1}{4}\|u_n\|^2-\frac{4-q}{4q}\lambda|f|_{\frac{6}{6-q}}S^{-\frac{q}{2}}\|u_n\|^q,
  \end{align*}
hence $\{u_n\}$ is bounded in $H^1(\R^3)$ by the fact that $1<q<2$. There exist a subsequence still denoted by $\{u_n\}$ and $u_0\in H^1(\R^3)$ such that
\begin{equation}\label{2.3a}
  \left\{
  \begin{array}{ll}
   u_n\rightharpoonup u_0~\text{in}~H^1(\R^3), \\
    u_n\rightarrow u_0~\text{in}~L^r_{loc}(\R^3),~\text{where}~1\leq r<6,\\
    u_n\rightarrow u_0~a.e.~\text{in}~\R^3.
  \end{array}
\right.
\end{equation}
The set $\R^3\bigcup\{\infty\}$ is compact for the stand topology which means that the measures can be identified as the dual space $C(\R^3\bigcup\{\infty\})$. For example, $\delta_\infty$ is well defined and $\delta_\infty(\varphi)=\varphi(\infty)$.

Due to the concentration-compactness principle in \cite{jaa,jb}, we can chose a subsequence denoted again by $\{u_n\}$ such that
\begin{equation}\label{2.3b}
  |\nabla u_n|^2\rightharpoonup d\mu\geq  |\nabla u_0|^2+\sum_{j\in \tilde{\Gamma}}\mu_j \delta_{x_j}+\mu_\infty\delta_\infty,
\end{equation}
\begin{equation}\label{2.3c}
  |u_n|^6\rightharpoonup d\nu= |u_0|^6+\sum_{j\in \tilde{\Gamma}}\nu_j \delta_{x_j}+\nu_\infty\delta_\infty,
\end{equation}
where $\delta_{x_j}$ and $\delta_\infty$ are the Dirac mass at $x_j$ and infinity respectively, and $x_j$ in the support of the measures $\mu, \nu$ and $\tilde{\Gamma}$ is an at most countable index set. What's more, by the Sobolev inequality we have
\begin{equation}\label{2.3d}
  \mu_j, \nu_j\geq 0,\ \ \mu_j\geq S\nu_j^{\frac{1}{3}}.
\end{equation}
We claim that $\nu_j=0$ for any $j\in \tilde{\Gamma}$. Arguing it by contradiction, for any $\epsilon>0$, let $\phi_j^\epsilon$ be a smooth cut-off function centered at $x_j$ such that
$0\leq \phi_j^\epsilon\leq 1$, $\phi_j^\epsilon\in C_0^{\infty}(B_\epsilon(x_j))$, $\phi_j^\epsilon=1$ in $B_{\frac{\epsilon}{2}}(x_j)$ and $|\nabla \phi_j^\epsilon|\leq \frac{4}{\epsilon}$. It is easy to see that
\begin{equation}\label{2.3e}
  \lim_{\epsilon\to 0}\int_{\R^3}|\nabla u_0|^2\phi_j^\epsilon dx=0.
\end{equation}
Indeed, by H\"{o}lder inequality and the Lebesgue Convergence Theorem we have
$$
0\leq \lim_{\epsilon\to 0}\int_{\R^3}|\nabla u_0|^2\phi_j^\epsilon dx=\lim_{\epsilon\to 0}\int_{B_\epsilon(x_j)}|\nabla u_0|^2\phi_j^\epsilon dx\leq \lim_{\epsilon\to 0}\int_{B_\epsilon(x_j)}|\nabla u_0|^2dx=0,
$$
thus \eqref{2.3e} holds. Similarly
\begin{equation}\label{2.3ee}
  \lim_{\epsilon\to 0}\int_{\R^3}|u_0|^2\phi_j^\epsilon dx=0,
\end{equation}
\begin{equation}\label{2.3eee}
  \lim_{\epsilon\to 0}\int_{\R^3}Q(x)|u_0|^6\phi_j^\epsilon dx=0,
\end{equation}
and
\begin{equation}\label{2.3eeee}
  \lim_{\epsilon\to 0}\int_{\R^3}f(x)|u_0|^2\phi_j^\epsilon dx=0.
\end{equation}
Some direct conclusions of \eqref{2.3e}-\eqref{2.3eeee} are
\begin{equation}\label{2.3f}
 \lim_{\epsilon\to 0}\lim_{n\to \infty}\int_{\R^3}|\nabla u_n|^2\phi_j^\epsilon dx\geq\lim_{\epsilon\to 0}\int_{\R^3}|\nabla u_0|^2\phi_j^\epsilon dx+\mu_j=\mu_j,
\end{equation}
\begin{equation}\label{2.3ff}
  \lim_{\epsilon\to 0}\lim_{n\to \infty}\int_{\R^3}| u_n|^2\phi_j^\epsilon dx=\lim_{\epsilon\to 0}\lim_{n\to \infty}\int_{B_{\epsilon}(x_j)}| u_n|^2\phi_j^\epsilon dx=\lim_{\epsilon\to 0}\int_{B_{\epsilon}(x_j)}| u_0|^2\phi_j^\epsilon dx=0,
\end{equation}
\begin{equation}\label{2.3fff}
  \begin{gathered}
\lim_{\epsilon\to 0}\lim_{n\to \infty}\int_{\R^3}Q(x)|u_n|^6\phi_j^\epsilon dx= \lim_{\epsilon\to 0}\lim_{n\to \infty}\int_{B_{\epsilon}(x_j)}Q(x)|u_n|^6\phi_j^\epsilon dx+Q(x_j)\nu_j\hfill\\
=\lim_{\epsilon\to 0}\int_{B_{\epsilon}(x_j)}| u_0|^6\phi_j^\epsilon dx+Q(x_j)\nu_j=Q(x_j)\nu_j,\hfill\\
\end{gathered}
\end{equation}
\begin{equation}\label{2.3ffff}
\begin{gathered}
\lim_{\epsilon\to 0}\lim_{n\to\infty}|\int_{\R^3}(\nabla u_n,\nabla\phi_j^\epsilon)u_ndx|\leq \lim_{\epsilon\to 0}\lim_{n\to\infty}\bigg(\int_{\R^3}|\nabla u_n|^2dx\bigg)^{\frac{1}{2}}\bigg(\int_{\R^3}u_n^2|\nabla\phi_j^\epsilon|^2dx\bigg)^{\frac{1}{2}}\hfill\\
\leq C\lim_{\epsilon\to 0}\lim_{n\to\infty}\bigg(\int_{\R^3}u_n^2|\nabla\phi_j^\epsilon|^2dx\bigg)^{\frac{1}{2}}=C\lim_{\epsilon\to 0}\lim_{n\to\infty}\bigg(\int_{B_{\epsilon}(x_j)}u_n^2|\nabla\phi_j^\epsilon|^2dx\bigg)^{\frac{1}{2}}\hfill\\
=C\lim_{\epsilon\to 0}\bigg(\int_{B_{\epsilon}(x_j)}u_0^2|\nabla\phi_j^\epsilon|^2dx\bigg)^{\frac{1}{2}}=C\lim_{\epsilon\to 0}\bigg(\int_{B_{\epsilon}(x_j)}u_0^2|\nabla\phi_j^\epsilon|^2dx\bigg)^{\frac{1}{2}}\hfill\\
\leq C\lim_{\epsilon\to 0}\bigg(\int_{B_{\epsilon}(x_j)}u_0^6dx\bigg)^{\frac{1}{6}}\bigg(\int_{B_{\epsilon}(x_j)}(\frac{4}{\epsilon})^3dx\bigg)^{\frac{1}{3}}=C\lim_{\epsilon\to 0}\bigg(\int_{B_{\epsilon}(x_j)}u_0^6dx\bigg)^{\frac{1}{6}}=0\hfill\\
\end{gathered}
\end{equation}
and
\begin{equation}\label{2.3fffff}
  \begin{gathered}
\lim_{\epsilon\to 0}\lim_{n\to \infty}\int_{\R^3}f(x)|u_n|^q\phi_j^\epsilon dx=\lim_{\epsilon\to 0}\lim_{n\to \infty}\int_{B_{\epsilon}(x_j)}f(x)|u_n|^q\phi_j^\epsilon dx\hfill\\
=\lim_{\epsilon\to 0}\int_{B_{\epsilon}(x_j)}f(x)|u_0|^q\phi_j^\epsilon dx=0.\hfill\\
\end{gathered}
\end{equation}
Since $\{u_n\}$ is bounded, $\lim_{\epsilon\to 0}\lim_{n\to \infty}\langle I^{'}(u_n),u_n\phi_j^\epsilon\rangle= 0$, that is
\begin{equation}\label{2.3g}
  \begin{gathered}
\lim_{\epsilon\to 0}\lim_{n\to \infty}\bigg[a\int_{\R^3}|\nabla u_n|^2\phi_j^\epsilon dx+a\int_{\R^3}(\nabla u_n,\nabla \phi_j^\epsilon)u_ndx+\int_{\R^3}u_n^2\phi_j^\epsilon dx\hfill\\
+b\int_{\R^3}|\nabla u_n|^2dx\int_{\R^3}|\nabla u_n|^2\phi_j^\epsilon dx+b\int_{\R^3}|\nabla u_n|^2dx\int_{\R^3}(\nabla u_n,\nabla\phi_j^\epsilon)u_ndx\hfill\\
-\int_{\R^3}u_n^6\phi_j^\epsilon dx-\lambda\int_{\R^3}f(x)|u_n|^q\phi_j^\epsilon dx\bigg]=0.\hfill\\
\end{gathered}
\end{equation}
By \eqref{2.3f}-\eqref{2.3g} we have $a\mu_j+ b\mu_j^2 \leq Q(x_j)\nu_j$, moreover $(\nu_j)^{\frac{1}{3}}\geq \frac{bS^2+\sqrt{b^2S^4+4aQ(x_j)S}}{2Q(x_j)}$ together with \eqref{2.3d}. Hence
\begin{align*}
 c+o(1)&=I(u_n)-\frac{1}{4}\langle I^{\prime}(u_n),u_n\rangle \\
       &=\frac{1}{4}\int_{\R^3}a|Du_n|^2+|u_n|^2dx+\frac{1}{12}\int_{\R^3}Q(x)|u_n|^6dx-\lambda(\frac{1}{q}-\frac{1}{4})\int_{\R^3}f|u_n|^qdx        \\
         & \geq \frac{a}{4}\mu_j+\frac{1}{12}Q(x_j)\nu_j+\frac{1}{4}\|u_0\|^2-\frac{4-q}{4q}\lambda|f|_{\frac{6}{6-q}}S^{-\frac{q}{2}}\|u_0\|^q\\
   & \geq \frac{abS^3}{4Q(x_j)}+\frac{b^3S^6}{24Q^2(x_j)}+\frac{(b^2S^4+4aQ(x_j)S)^{\frac{3}{2}}}{24Q^2(x_j)}-C_1\lambda^{\frac{2}{2-q}}\\
& \geq \frac{abS^3}{4|Q|_{\infty}}+\frac{b^3S^6}{24|Q|^2_{\infty}}+\frac{(b^2S^4+4a|Q|_{\infty}S)^{\frac{3}{2}}}{24|Q|^2_{\infty}}-C_1\lambda^{\frac{2}{2-q}},
  \end{align*}
where $C_1=\frac{2-q}{4q}\bigg[\frac{(4-q)|f|_{\frac{6}{6-q}}}{2S^{\frac{q}{2}}}\bigg]^{\frac{2}{2-q}}\in (0,+\infty)$, which is a contradiction! Therefore $\nu_j=0$ for any $j\in \tilde{\Gamma}$.

We now study the concentration at infinity, we define a new cut-off function $\varphi\in C_0^\infty\big(\R^3,[0,1]\big)$, such that $\varphi(x)=0$ if $|x|<R$, $\varphi(x)=1$ if $|x|>R$ and $|\nabla \varphi|\leq \frac{2}{R}$. Consider the following equalities
$$
\mu_\infty=\lim_{R\to\infty}\lim_{n\to\infty}\sup\int_{|x|>R}|\nabla u_n|^2\varphi dx
$$
and
$$
\nu_\infty=\lim_{R\to\infty}\lim_{n\to\infty}\sup\int_{|x|>R}|u_n|^6\varphi dx.
$$
Using the same technique as at the $x_j$, we obtain the same conclusion, namely, $\mu_\infty=\nu_\infty=0$.

Then in view of \eqref{2.3c}, we have $u_n\to u_0$ in $L^6(\R^3)$.
Next we will show that $u_n\to u_0$ in $H^1(\R^3)$. In fact, set
\begin{equation}\label{0.0a}
  v_n:=u_n-u_0,
\end{equation}
\begin{equation}\label{0.0b}
  I_{11}(u_n):=(a+b\int_{\R^3}|\nabla u_n|^2dx)\int_{\R^3}\big(\nabla u_n,\nabla v_n\big)dx,
\end{equation}
and
\begin{equation}\label{0.0c}
  I_{22}(u_n):=(a+b\int_{\R^3}|\nabla u_0|^2dx)\int_{\R^3}\big(\nabla u_0,\nabla v_n\big)dx.
\end{equation}
From the fact $v_n=u_n-u_0\to 0$ and $u_n\to u_0$ in $L^6(\R^3)$ we have
\begin{equation}\label{0.0d}
  \int_{\R^3}|\nabla u_n|^2dx-\int_{\R^3}|\nabla u_0|^2dx=\int_{\R^3}|\nabla v_n|^2dx+o(1),
\end{equation}
\begin{equation}\label{0.0e}
  I_{2}(u_n):= \int_{\R^3}(|u_n|^4u_n-|u_0|^4u_0)(u_n-u_0)dx=o(1),
\end{equation}
\begin{equation}\label{0.0f}
  I_{3}(u_n):=\int_{\R^3}f(x)(|u_n|^{q-2}u_n-|u_0|^{q-2}u_0)(u_n-u_0)dx=o(1).
\end{equation}
Then combing with \eqref{0.0a}-\eqref{0.0d} we have
\begin{align*}
 I_1(u_n)&:= I_{11}(u_n)-I_{22}(u_n) \\
       &=(a+b\int_{\R^3}|\nabla u_n|^2dx)\int_{\R^3}|\nabla v_n|^2dx+b\int_{\R^3}|\nabla v_n|^2dx\int_{\R^3}(\nabla u_0,\nabla v_n)dx+o(1)\\
         & \geq a\int_{\R^3}|\nabla v_n|^2dx+b\int_{\R^3}|\nabla v_n|^2dx\int_{\R^3}(\nabla u_0,\nabla v_n)dx+o(1)\\
   & = a\int_{\R^3}|\nabla v_n|^2dx+o(1).\tag{2.40}
  \end{align*}
Therefore \eqref{0.0e}-$(2.40)$ imply that
\begin{align*}
 o(1)&= \langle I^{\prime}(u_n)-I^{\prime}(u_0),v_n\rangle \\
       &=I_{1}(u_n)+\int_{\R^3}|v_n|^2dx-I_{2}(u_n)-I_{3}(u_n)\\
         &\geq a\int_{\R^3}|\nabla v_n|^2dx+\int_{\R^3}|v_n|^2dx+o(1)\\
   & = \|v_n\|^2+o(1),
  \end{align*}
and then $\|v_n\|\to 0$ which is $u_n\to u_0$ in $H^1(\R^3)$ and we complete the proof.
\end{proof}
\vskip0.5cm

\section{The proof of Theorem \ref{main theorem}}
From this section, we can find out that why we can see the Theorem \ref{corollary} as a corollary of the Theorem \ref{main theorem}. Before we find the two different positive solutions of problem \eqref{g}, we introduce the following well-known proposition:
\begin{proposition}\label{propo}
(Ekeland's variational principle \cite{j}, Theorem 1.1)
Let $V$ be a complete metric space and $F:V\to \R\cup\{+\infty\}$ be lower semicontinuous, bounded from below. Then for any $\epsilon>0$, there exists some point $v\in V$ with
$$
F(v)\leq \inf_VF+\epsilon,\ \ F(w)\geq F(v)-\epsilon d(v,w)\ \ \text{for all}\ \  w\in V.
$$
\end{proposition}
Now, we will split it into three parts to prove Theorem \ref{main theorem}.
\subsection{Existence of the first positive solution}

If $\lambda\in (0,\lambda_0)$, then by Lemma \ref{2.1} there exists a $(PS)_c$ sequence $\{u_n\}$ of the functional $I(u)$. Lemma \ref{2.2} and Lemma \ref{2.3} tell us that $u_n\to u_0$ in $H^1(\R^3)$, then $I^{\prime}(u_0)=0$ and $I(u_0)=c>0$ which imply that $u_0$ is nontrivial.

Obviously, $|u_0|$ is also a solution of the problem \eqref{g}, then by maximum principle, $u_0$ is positive and then $u_0$ is a positive solution of the problem \eqref{g}.
\subsection{Existence of the second positive solution}
We investigate the second positive solution for problem \eqref{g} similar to an idea from \cite{jc}.
For $\rho>0$ given by Lemma \ref{2.1}(i), define
$$
\overline{B}_\beta=\{u\in H^1(\R^3),\|u\|\leq \beta\},\ \ \partial B_\beta=\{u\in H^1(\R^3),\|u\|= \beta\}
$$
and clearly $\overline{B}_\beta$ is a complete metric space with the distance
$$
d(u,v)=\|u-v\|, \ \forall u,v\in \overline{B}_\beta.
$$
Lemma \ref{2.1} tells us that
\begin{equation}\label{3.1a}
 I(u)|_{\partial B_\beta}\geq \eta>0.
\end{equation}
It's obvious that $I(u)$ is lower semicontinuous and bounded from below on $\overline{B}_\beta$. We claim that
\begin{equation}\label{3.1b}
  c_1:=\inf_{u\in \overline{B}_\beta}I(u)<0.
\end{equation}
Indeed, we chose a nonnegative function $\psi\in C_0^{\infty}(\R^3)$, and clearly $\psi\in H^1(\R^3)$. Since $1<q<2$, we have
\begin{align*}
\lim_{t\to 0}\frac{I(t\psi)}{t^q}&=\lim_{t\to 0}\frac{\frac{t^2}{2}\|\psi\|^2+\frac{bt^4}{4}\bigg(\int_{\R^3}|\nabla \psi|^2dx\bigg)^2-\frac{t^6}{6}\int_{\R^3}Q(x)|\psi|^6dx-\frac{\lambda t^q}{q}\int_{\R^3}f|\psi|^qdx}{t^q} \\
       &=-\frac{\lambda}{q}\int_{\R^3}f(x)|\psi|^qdx<0.
\end{align*}
Therefore there exists a sufficiently small $t_0>0$ such that $\|t_0\psi\|\leq \beta$ and $I(t_0\psi)<0$, which imply that \eqref{3.1b} holds.

By Proposition \ref{propo}, for any $n\in N$ there exists $\tilde{u}_n$ such that
\begin{equation}\label{3.1c}
 c_1\leq I(\tilde{u}_n)\leq c_1+\frac{1}{n},
\end{equation}
and
\begin{equation}\label{3.1d}
  I(v)\geq I(\tilde{u}_n)-\frac{1}{n}\|\tilde{u}_n-v\|, \ \  \forall v\in  \overline{B}_\beta.
\end{equation}
Firstly, we claim that $\|\tilde{u}_n\|<\beta$ for $n\in N$ sufficiently large. In fact, we will argue it by contradiction and just suppose that $\|\tilde{u}_n\|=\beta$ for infinitely many $n$, without loss of generality, we may assume that $\|\tilde{u}_n\|=\beta$ for any $n\in N$. It follows from \eqref{3.1a} that
$$
I(\tilde{u}_n)\geq \eta>0,
$$
then combing it with \eqref{3.1c}, we have $c_1\geq \eta>0$ which is a contradiction to \eqref{3.1b}.

Next, we will show that $I^{\prime}(\tilde{u}_n)\to 0$ in $(H^1(\R^3))^{*}$. Indeed, set
$$
v_n=\tilde{u}_n+tu, \ \  \forall u\in B_1=\{u\in H^1(\R^3),\|u\|=1\},
$$
where $t>0$ small enough such that $2t+t^2\leq \rho^2-\|\tilde{u}_n\|^2$ for fixed $n$ large, then
\begin{align*}
\|v_n\|^2&=\|\tilde{u}_n\|^2+2t(\tilde{u}_n,u)+t^2 \\
       &\leq \|\tilde{u}_n\|^2+2t+t^2\\
         &\leq \beta^2
\end{align*}
which imply that $v_n\in \overline{B}_\beta$. So it follows from \eqref{3.1d} that
$$
I(v_n)\geq I(\tilde{u}_n)-\frac{t}{n}\|\tilde{u}_n-v_n\|,
$$
that is,
$$
\frac{I(\tilde{u}_n+tu)-I(\tilde{u}_n)}{t}\geq -\frac{1}{n}.
$$
Letting $t\to 0$, then we have $\langle I^{\prime}(\tilde{u}_n), u\rangle\geq -\frac{1}{n}$ for any fixed $n$ large. Similarly, chose $t<0$ and $|t|$ small enough, repeating the process above we have $\langle I^{\prime}(\tilde{u}_n), u\rangle\leq \frac{1}{n}$ for any fixed $n$ large. Therefore the conclusion
$$
I^{\prime}(\tilde{u}_n), u\rangle\to 0\ \ \text{as}\ \ n\to\infty, \ \  \forall u\in B_1
$$
implies that $I^{\prime}(\tilde{u}_n)\to 0$ in $(H^1(\R^3))^{*}$.

Finally, we know that $\{\tilde{u}_n\}$ is a $(PS)_{c_1}$ sequence for the functional $I(u)$ and by the Lemma \ref{2.3}, there exists $u_1$ such that $\tilde{u}_n\to u_1$ in $H^1(\R^3)$ and then $I^{\prime}(u_1)=0$, that is $u_1$ a solution of problem \eqref{g} with $I(u_1)=c_1<0$. Moreover, the strong maximum principle implies that $u_1$ is positive.
\subsection{Existence of the positive ground state solution}
To obtain a positive ground state solution for problem \eqref{g}, we define
$$
m:= \inf_{\mathcal{N}}I(u),
$$
where $\mathcal{N}:=\{u\in H^1(\R^3):I^{\prime}(u)=0\}$. Firstly we have the following claim:
\vskip0.3cm
\underline{\textbf{Claim 1:}} $\mathcal{N}\neq\emptyset$ and $m\in(-\infty, c]$.
\vskip0.3cm
\underline{\textbf{Proof of the Claim 1:}} It's obvious that $\mathcal{N}\neq\emptyset$ and then $m\leq \max\{I(u_0),I(u_1)\}=c$.

On the other hand,
$\forall u\in \mathcal{N}$, we have
\begin{equation}\label{3.1e}
 {\begin{split}
    I(u)&= I(u)-\frac{1}{4}\langle I^{\prime}(u),u\rangle \\
       &=\frac{1}{4}\|u\|^2+\frac{1}{12}\int_{\R^3}Q(x)|u|^6dx-\lambda(\frac{1}{q}-\frac{1}{4})\int_{\R^3}f|u|^qdx        \\
         & \geq \frac{1}{4}\|u\|^2-\lambda\frac{4-q}{4q}\lambda|f|_{\frac{6}{6-q}}S^{-\frac{q}{2}}\|u\|^q,
  \end{split}
 }
\end{equation}
hence $I(u)$ is coercive and bounded below on $\mathcal{N}$ by the fact that $1<q<2$, that is $m>-\infty$.

By virtue of Claim 1, we can choose a minimizing sequence of $m$, that is a sequence $\{v_n\}\subset \mathcal{N}$ satisfying
$$
I(v_n)\to m\ \  \text{as} \ \ n\to\infty\ \   \text{and}  \ \ I^{\prime}(v_n)=0.
$$
Therefore $\{v_n\}$ is a $(PS)_m$ sequence of $I(u)$ with $m<\frac{abS^3}{4|Q|_{\infty}}+\frac{b^3S^6}{24|Q|^2_{\infty}}+\frac{(b^2S^4+4a|Q|_{\infty}S)^{\frac{3}{2}}}{24|Q|^2_{\infty}}-C_0\lambda^{\frac{2}{2-q}}$. Then similar to $(3.5)$, $v_n$ is bounded in $H^1(\R^3)$ and there exists a $v\in H^1(\R^3)$ such that $v_n\rightharpoonup v$ in $H^1(\R^3)$.
\vskip0.3cm
\underline{\textbf{Claim 2:}} $v\neq 0$.
\vskip0.3cm
\underline{\textbf{Proof of the Claim 2:}} Argue it by contradiction and just suppose that $v\equiv 0$, hence
\begin{equation}\label{h1}
 \int_{\R^3}f(x)|v_n|^qdx\to 0
\end{equation}
by $f(x)\in L^{\frac{6}{6-q}}$ and $v_n\to 0$ in $L_{loc}^r(\R^3)$ ($1\leq r<6$). Therefore we can infer from \eqref{h1} that
\begin{equation}\label{h2}
  \|v_n\|^2+b\bigg(\int_{\R^3}|\nabla v_n|^2dx\bigg)^2-\int_{\R^3}Q(x)|v_n|^6dx=o(1),
\end{equation}
and
\begin{equation}\label{h3}
  \frac{1}{2}\|v_n\|^2+\frac{b}{4}\bigg(\int_{\R^3}|\nabla v_n|^2dx\bigg)^2-\frac{1}{6}\int_{\R^3}Q(x)|v_n|^6dx=m+o(1).
\end{equation}
Let
\begin{equation}\label{h4}
  \|v_n\|^2\to l_1\ \ \text{and}\ \ b\bigg(\int_{\R^3}|\nabla v_n|^2dx\bigg)^2\to l_2,
\end{equation}
thus \eqref{h2}-\eqref{h4} give us that
\begin{equation}\label{h5}
   m=\frac{l_1}{3}+\frac{l_2}{12}
\end{equation}
and
\begin{equation}\label{h6}
  |Q|_\infty\int_{\R^3}|v_n|^6dx\geq\int_{\R^3}Q(x)|v_n|^6dx\to l_1+l_2.
\end{equation}
By the Sobolev imbedding theorem,
\begin{equation}\label{h7}
  \|v_n\|^2\geq aS\bigg(\int_{\R^3}|v_n|^6dx\bigg)^{\frac{1}{3}},
\end{equation}
and
\begin{equation}\label{h8}
  b\bigg(\int_{\R^3}|\nabla v_n|^2dx\bigg)^2\geq bS^2\bigg(\int_{\R^3}|v_n|^6dx\bigg)^{\frac{2}{3}},
\end{equation}
It follows from \eqref{h6}-\eqref{h8} that,
\begin{equation}\label{h9}
  l_1\geq aS\bigg(\frac{l_1+l_2}{|Q|_\infty}\bigg)^{\frac{1}{3}} \ \ \text{and}\ \ l_2\geq bS^2\bigg(\frac{l_1+l_2}{|Q|_\infty}\bigg)^{\frac{2}{3}}.
\end{equation}
Now we define a function
$$
h(t)=|Q|_\infty ^{\frac{2}{3}}t^2-bS^2t-aS|Q|_\infty ^{\frac{1}{3}}.
$$
It's obvious that there exist $t_1<0<t_2$ such that $h(t_1)=h(t_2)=0$, and
$$
t_1=\frac{bS^2-\sqrt{b^2S^4+4a|Q|_\infty S}}{2|Q|_\infty ^{\frac{2}{3}}}, \  \ t_2=\frac{bS^2+\sqrt{b^2S^4+4a|Q|_\infty S}}{2|Q|_\infty ^{\frac{2}{3}}}.
$$
Therefore if $h(t)\geq 0$, then $t\leq t_1$ or $t\geq t_2$. On the other hand,
$$
h\bigg((l_1+l_2)^{\frac{1}{3}}\bigg)= |Q|_\infty ^{\frac{2}{3}}(l_1+l_2)^{\frac{2}{3}}-bS^2(l_1+l_2)^{\frac{1}{3}}-aS|Q|_\infty ^{\frac{1}{3}}\stackrel{\mathrm{\eqref{h9}}}{\geq }0,
$$
hence
\begin{equation}\label{h10}
 (l_1+l_2)^{\frac{1}{3}}\geq t_2=\frac{bS^2+\sqrt{b^2S^4+4a|Q|_\infty S}}{2|Q|_\infty ^{\frac{2}{3}}}.
\end{equation}
As a direct conclusion of \eqref{h5} and \eqref{h9}-\eqref{h10}, we have
\begin{equation}\label{h11}
  {\begin{split}
     m &= \frac{l_1}{3}+\frac{l_2}{12}\geq \frac{1}{3}aS(l_1+l_2)^{\frac{1}{3}}+  \frac{1}{12}bS^2(l_1+l_2)^{\frac{2}{3}}\\
       & \geq \frac{abS^3}{4|Q|_{\infty}}+\frac{b^3S^6}{24|Q|^2_{\infty}}+\frac{(b^2S^4+4a|Q|_{\infty}S)^{\frac{3}{2}}}{24|Q|^2_{\infty}}>c,
   \end{split}
  }
\end{equation}
a contradiction! So $v\neq 0$.

Similar to the Lemma \ref{2.3}, $v_n\to v$ in $H^1(\R^3)$, then $I(v)=m$ and $I^{\prime}(v)=0$. Obviously, $|v|$ is also a critical point, by strong maximum principle $v$ is positive. Hence $v$ is a nontrivial ground state solution of problem \eqref{g}.
\vskip0.5cm
\section{ The convergence property of the ground state solution.}
In this section, we give the proof of Theorem \ref{theorem2}. Firstly, the energy functional of problem \eqref{g} is written as
$$
\begin{gathered}
I_b(u)=\frac{1}{2}\int_{\R^3}a|\nabla u|^2+|u|^2dx+\frac{b}{4}\bigg(\int_{\R^3}|\nabla u|^2dx\bigg)^2\hfill\\
-\frac{1}{6}\int_{\R^3}Q(x)|u|^6dx-\frac{\lambda}{q}\int_{\R^3}f(x)|u|^qdx,
\end{gathered}
$$
and the least energy of which is
$$
m_b:= \inf_{\mathcal{N}_b}I_b(u),
$$
where ${\mathcal{N}_b}:=\{u\in H^1(\R^3):I^{\prime}_b(u)=0\}$.
\vskip0.3cm
\noindent\textbf{Proof of the Theorem \ref{theorem2}:} For any fixed $\lambda\in(0,\lambda_0)$ and any $b>0$, $u_b\in \mathcal{N}_b$ is a positive ground solution of problem \eqref{g} obtained in Section 3 and satisfies $I_b(u_b)=m_b$. Thus for any $b\in (0,1]$,
\begin{align*}
 m_b&<\frac{abS^3}{4|Q|_{\infty}}+\frac{b^3S^6}{24|Q|^2_{\infty}}+\frac{(b^2S^4+4a|Q|_{\infty}S)^{\frac{3}{2}}}{24|Q|^2_{\infty}}\\    & \leq \frac{aS^3}{4|Q|_{\infty}}+\frac{S^6}{24|Q|^2_{\infty}}+\frac{(S^4+4a|Q|_{\infty}S)^{\frac{3}{2}}}{24|Q|^2_{\infty}}:=M<+\infty,
  \end{align*}
where $M$ is independent on $b$.

On one hand,
\begin{align*}
M&>I_{b_n}(u_{b_n})-\frac{1}{4}\langle I^{\prime}_{b_n}(u),u_{b_n}\rangle \\
       &=\frac{1}{4}\|u_{b_n}\|^2+\frac{1}{12}\int_{\R^3}Q(x)|u_{b_n}|^6dx-\lambda(\frac{1}{q}-\frac{1}{4})\int_{\R^3}f|u_{b_n}|^qdx        \\
         & \geq \frac{1}{4}\|u_{b_n}\|^2-\frac{4-q}{4q}\lambda|f|_{\frac{6}{6-q}}S^{-\frac{q}{2}}\|u_{b_n}\|^q.
  \end{align*}
This shows that $\{u_{b_n}\}$ is bounded in $H^1(\R^3)$ due to $1<q<2$. Hence, there exist a subsequence still denoted $\{b_n\}$ and $w_0\in H^1(\R^3)$ such that $u_{b_{n}}\rightharpoonup w_0$ in $H^1(\R^3)$. Similar to Section 3, we have $u_{b_{n}}\to w_0\neq 0$ in $H^1(\R^3)$.

On the other hand, $\forall \varphi\in C_0^{\infty}(H^1(\R^3))$, we have
$$
\begin{gathered}
\langle I^{\prime}_{b_n}(u_{b_n}),\varphi\rangle=\int_{\R^3}a\nabla u_{b_n}\nabla \varphi+u_{b_n}\varphi dx+b_n\int_{\R^3}|\nabla u_{b_n}|^2dx\int_{\R^3}\nabla u_{b_n}\nabla \varphi dx\hfill\\
-\int_{\R^3}Q(x)|u_{b_n}|^4u_{b_n}\varphi dx-\lambda\int_{\R^3}f|u_{b_n}|^{q-1}u_{b_n}\varphi dx\hfill\\
\to \int_{\R^3}a\nabla w_0\nabla \varphi+w_0\varphi dx
-\int_{\R^3}Q(x)|w_0|^4w_0\varphi dx-\lambda\int_{\R^3}f(x)|w_0|^{q-1}w_0\varphi dx\hfill\\
\end{gathered}
$$
as $b_n\to 0$ and $n\to \infty$. By the fact $I^{\prime}_{b_n}(u_{b_n})=0$ we have
\begin{gather*}
 \int_{\R^3}a\nabla w_0\nabla \varphi+w_0\varphi dx
-\int_{\R^3}Q(x)|w_0|^4w_0\varphi dx-\lambda\int_{\R^3}f(x)|w_0|^{q-2}w_0\varphi dx=0
\end{gather*}
for any $\varphi\in H^1(\R^3)$. Therefore $w_0$ is a solution of the problem \eqref{h}.\hspace{\stretch{1}}$\Box$
\vskip0.5cm
\section{ The proof of Theorem \ref{theorem3}.}
We will give some Lemmas before we prove the Theorem \ref{theorem3}. We just give the detail proof of the difference from Lemmas \ref{2.1}-\ref{2.3}.

The following Lemma has a great difference which seems failure in \cite{ckw} from our Lemma \ref{2.1}.

\begin{lemma}\label{5.1}
For any $\lambda\in (0, \tilde{\lambda}_0)$, the functional $I(u)$ of \eqref{g} satisfies the Mount-pass geometry around $0\in H^1(\R^3)$, that is,

$(i)$ there exist $\tilde{\eta},\tilde{\beta}>0$ such that $I(u)\geq \tilde{\eta}>0$ when $\|u\|=\tilde{\beta}$;

$(ii)$ there exists $\tilde{e}\in H^1(\R^3)$ with $\|\tilde{e}\|>\tilde{\beta}$ such that $I(\tilde{e})<0$.
\end{lemma}

\begin{proof}
 $(i)$ By the definition of $I(u)$,
\begin{align*}
 I(u)&= \frac{1}{2}\|u\|^2+\frac{b}{4}\bigg(\int_{\R^3}|\nabla u|^2dx\bigg)^2
-\frac{1}{6}\int_{\R^3}Q(x)|u|^6dx-\frac{\lambda}{q}\int_{\R^3}f(x)|u|^qdx                                                                      \\
 &\geq \frac{b}{4}\bigg(\int_{\R^3}|\nabla u|^2dx\bigg)^2
-\frac{1}{6}\int_{\R^3}Q(x)|u|^6dx-\frac{\lambda}{q}\int_{\R^3}f(x)|u|^qdx         \\
         & \geq \frac{b}{4}\bigg(\int_{\R^3}|\nabla u|^2dx\bigg)^2-\frac{|Q|_\infty}{6S^3}\bigg(\int_{\R^3}|\nabla u|^2dx\bigg)^3-\frac{\lambda |f|_{\frac{6}{6-q}}}{qS^{\frac{q}{2}}}\bigg(\int_{\R^3}|\nabla u|^2dx\bigg)^{\frac{q}{2}}\\
&=t^q\bigg(\frac{b}{4}t^{4-q}-\frac{|Q|_\infty}{6S^3}t^{6-q}-\frac{\lambda |f|_{\frac{6}{6-q}}}{qS^{\frac{q}{2}}}\bigg),
  \end{align*}
where $t:=\bigg(\int_{\R^3}|\nabla u|^2dx\bigg)^{\frac{1}{2}}$. Therefore letting $\tilde{\lambda}_0=C_2b^{\frac{6-q}{2}}$, here $C_2:=\frac{qS^{\frac{q}{2}}}{2(6-q)|f|_{\frac{6}{6-q}}}\bigg(\frac{3(4-q)S^3}{2(6-q)}\bigg)^{\frac{4-q}{2}}$, and then there exist $\tilde{\eta},\tilde{\beta}>0$ such that $I(u)\geq \tilde{\eta}>0$ when $\|u\|=\tilde{\beta}$ for any $\lambda\in (0,\tilde{\lambda_0})$.

$(ii)$ The proof is similar to Lemma \ref{2.1}$(ii)$, so we omit it here.
\end{proof}
Now by Lemma \ref{5.1}, we can construct a $(PS)$ sequence of the functional $I(u)$ at the level
\begin{equation}\label{5.1a}
\tilde{ c}:=\inf_{\gamma\in \Gamma}\sup_{t\in [0,1]}I(\gamma(t)),
\end{equation}
where
\begin{equation}\label{5.1b}
  \Gamma:=\{\gamma\in C([0,1],H^1(\R^3)):\gamma(0)=0, I(\gamma(1)<0\},
\end{equation}
 that is, there exists a sequence $\{u_n\}\subset H^1(\R^3)$ satisfies
\begin{equation}\label{5.1c}
  I(u_n)\to \tilde{c},\ \ I^{\prime}(u_n)\to 0\ \  \text{as}\ \ n\to \infty.
\end{equation}
\vskip0.3cm
\begin{lemma}\label{5.2}
Assume $1<q<2$ and $\lambda\in(0,\tilde{\lambda}_0)$, then the critical energy
$$
\tilde{c}<\frac{abS^3}{4|Q|_{\infty}}+\frac{b^3S^6}{24|Q|^2_{\infty}}+\frac{(b^2S^4+4a|Q|_{\infty}S)^{\frac{3}{2}}}{24|Q|^2_{\infty}}-C_3\lambda^{\frac{2}{2-q}},
$$
where $C_3=\frac{2-q}{2}\bigg(\frac{3q}{2a}\bigg)^{\frac{q}{2-q}}\bigg(\frac{(6-q)|f|_{\frac{6}{6-q}}}{6qS^{\frac{q}{2}}}\bigg)^{\frac{2}{2-q}}$ is a positive constant.
\end{lemma}
\vskip0.3cm
\begin{proof}
The proof is totally same as Lemma \ref{2.2} except changing the constant $C_2$ to $C_3$ in \eqref{2.1n}, \eqref{2.1nn}, and \eqref{2.1nnn}.
\end{proof}
\vskip0.3cm
\begin{lemma}\label{5.3}
The functional $I(u)$ satisfies the $(PS)_{\tilde{c}}$ condition when $\lambda\in (0,\tilde{\lambda}_0)$ and
$$
\tilde{c}<\frac{abS^3}{4|Q|_{\infty}}+\frac{b^3S^6}{24|Q|^2_{\infty}}+\frac{(b^2S^4+4a|Q|_{\infty}S)^{\frac{3}{2}}}{24|Q|^2_{\infty}}-C_3\lambda^{\frac{2}{2-q}}.
$$
\end{lemma}
\vskip0.3cm
\begin{proof}
For $\lambda\in (0,\tilde{\lambda}_0)$, by Lemma \ref{5.1}, there exists a sequence $\{u_n\}$ assuming that
$$
I(u_n)\to \tilde{c},\ \ I^{\prime}(u_n)\to 0\ \  \text{as}\ \ n\to \infty.
$$
Then similar to Lemma \ref{2.1}, $\{u_n\}$ is bounded in $H^1(\R^3)$
and there exist a subsequence still denoted by $\{u_n\}$ and $u\in H^1(\R^3)$ such that
\begin{equation}\label{5.4a}
  \left\{
  \begin{array}{ll}
   u_n\rightharpoonup u~\text{in}~H^1(\R^3), \\
    u_n\rightarrow u~\text{in}~L^r_{loc}(\R^3),~\text{where}~1\leq r<6,\\
    u_n\rightarrow u~a.e.~\text{in}~\R^3.
  \end{array}
\right.
\end{equation}
Set
\begin{equation}\label{5.4b}
  A^2:=\lim\limits_{n\to\infty}\int_{\R^3}|\nabla u_n|^2dx
\end{equation}
and
\begin{equation}\label{5.4c}
  v_n:=u_n-u.
\end{equation}
We can infer from the Brezis-Lieb Lemma in \cite{jd} that
\begin{equation}\label{5.4d}
 A^2=\int_{\R^3}|\nabla u_n|^2dx+o(1)=\int_{\R^3}|\nabla v_n|^2dx+\int_{\R^3}|\nabla u|^2dx+o(1),
 \end{equation}
\begin{equation}\label{5.4dd}
  \int_{\R^3}|u_n|^2dx=\int_{\R^3}|v_n|^2dx+\int_{\R^3}|u|^2dx+o(1),
\end{equation}
\begin{equation}\label{5.4ddd}
  \int_{\R^3}Q(x)|u_n|^6dx=\int_{\R^3}Q(x)|v_n|^6dx+\int_{\R^3}Q(x)|u|^6dx+o(1),
\end{equation}
and since $f\in L^{\frac{6}{6-q}}$, we have
\begin{equation}\label{5.4dddd}
  \int_{\R^3}f(x)|u_n|^qdx\stackrel{\mathrm{\eqref{5.4a}}}{=}\int_{\R^3}f(x)|u|^qdx+o(1).
\end{equation}
Now we introduce a new functional, for any $u\in H^1(\R^3)$,
\begin{equation}\label{5.4e}
  J(u):=\frac{a+bA^2}{2}\int_{\R^3}|\nabla u|^2dx+\frac{1}{2}\int_{\R^3}|u|^2dx-\frac{1}{6}\int_{\R^3}Q(x)|u|^6dx-\frac{\lambda}{q}\int_{\R^3}f(x)|u|^q.
\end{equation}
Hence
\begin{equation}\label{5.4f}
  {\begin{split}
     J(u_n) & =I(u_n)+\frac{bA^2}{2}\int_{\R^3}|\nabla u_n|^2dx-\frac{b}{4}\bigg(\int_{\R^3}|\nabla u_n|^2dx\bigg)^2 \\
       &\stackrel{\mathrm{\eqref{5.4d}}}{=}\tilde{c}+\frac{bA^4}{2}-\frac{bA^4}{4}+o(1)=\tilde{c}+\frac{bA^4}{4}+o(1)
   \end{split}
  }
\end{equation}
and for any $\varphi\in H^1(\R^3)$ we have
\begin{align*}
  \langle J^{\prime}(u_n),\varphi\rangle & =\langle I^{\prime}(u_n),\varphi\rangle+bA^2\int_{\R^3}\nabla u_n \nabla \varphi dx-b\int_{\R^3}|\nabla u_n|^2dx\int_{\R^3}\nabla u_n\nabla \varphi dx \\
   & =b(A^2-\int_{\R^3}|\nabla u_n|^2dx)\int_{\R^3}\nabla u_n\nabla \varphi dx\\
   &=b\bigg(\int_{\R^3}\nabla u_n\nabla \varphi dx\bigg)o(1)=o(1)
\end{align*}
show that $\{u_n\}$ is a bounded $(PS)_{c+\frac{bA^4}{4}}$ sequence of $J(u)$. In particular
\begin{eqnarray*}
o(1) &=& \langle J^{\prime}(u_n),\varphi\rangle \\
   &=& (a+bA^2)\int_{\R^3}\nabla u_n\nabla \varphi dx+\int_{\R^3} u_n \varphi dx-\int_{\R^3}Q(x)|u_n|^4u_n\varphi dx \\
   &&-\lambda\int_{\R^3}f(x)|u_n|^{q-2}u_n\varphi \\
   &\stackrel{\mathrm{\eqref{5.4a}}}{=}& (a+bA^2)\int_{\R^3}\nabla u\nabla \varphi dx+\int_{\R^3} u \varphi dx-\int_{\R^3}Q(x)|u|^4u\varphi\\
   &&-\lambda\int_{\R^3}f(x)|u|^{q-2}u\varphi+o(1)\\
   &=&\langle J^{\prime}(u),\varphi\rangle+o(1)
\end{eqnarray*}
implies that $\langle J^{\prime}(u),u\rangle=0$ by taking $u=\varphi$ since $\varphi$ is arbitrary. Therefore
\begin{equation}\label{5.4ff}
  {\begin{split}
     J(u) & = J(u)-\frac{1}{6}\langle J^{\prime}(u),u\rangle \\
       & =\frac{a+bA^2}{3}\int_{\R^3}|\nabla u|^2dx+\frac{1}{3}\int_{\R^3}|u|^2dx-\lambda(\frac{1}{q}-\frac{1}{6})\int_{\R^3}f(x)|u|^qdx\\
       &\geq \frac{a}{3}\int_{\R^3}|\nabla u|^2dx- \frac{\lambda(6-q)|f|_{\frac{6}{6-q}}}{6qS^{\frac{q}{2}}}\bigg(\int_{\R^3}|\nabla u|^2dx\bigg)^{\frac{q}{2}}+\frac{bA^2}{4}\int_{\R^3}|\nabla u|^2dx\\
       &\geq -C_3\lambda^{\frac{2}{2-q}}+\frac{bA^2}{4}\int_{\R^3}|\nabla u|^2dx,
   \end{split}
  }
\end{equation}
where $C_3=\frac{2-q}{2}\bigg(\frac{3q}{2a}\bigg)^{\frac{q}{2-q}}\bigg(\frac{\lambda(6-q)|f|_{\frac{6}{6-q}}}{6qS^{\frac{q}{2}}}\bigg)^{\frac{2}{2-q}}$. Since $\{u_n\}$ is bounded and then,
\begin{align*}
  o(1) & = \langle J^{\prime}(u_n),u_n\rangle\\
   &\stackrel{\mathrm{\eqref{5.4d}}}{=}\langle J^{\prime}(u),u\rangle+(a+bA^2)\int_{\R^3}|\nabla v_n|^2dx+\int_{\R^3}|v_n|^2dx-\int_{\R^3}Q(x)|v_n|^6dx\\
   &=(a+bA^2)\int_{\R^3}|\nabla v_n|^2dx+\int_{\R^3}|v_n|^2dx-\int_{\R^3}Q(x)|v_n|^6dx\\
   &\stackrel{\mathrm{\eqref{5.4d}}}{=}\|v_n\|^2+b\bigg(\int_{\R^3}|\nabla v_n|^2dx\bigg)^2+b\int_{\R^3}|\nabla v_n|^2dx\int_{\R^3}|\nabla u|^2dx-\int_{\R^3}Q(x)|v_n|^6dx,
\end{align*}
which is
\begin{gather}\label{5.4g}
         \|v_n\|^2+b\bigg(\int_{\R^3}|\nabla v_n|^2dx\bigg)^2+b\int_{\R^3}|\nabla v_n|^2dx\int_{\R^3}|\nabla u|^2dx-\int_{\R^3}Q(x)|v_n|^6dx=o(1).
\end{gather}
Without loss of generality, we can assume
$$
\|v_n\|^2\to \tilde{l}_1,b\bigg(\int_{\R^3}|\nabla v_n|^2dx\bigg)^2+b\int_{\R^3}|\nabla v_n|^2dx\int_{\R^3}|\nabla u|^2dx\to \tilde{l}_2, \int_{\R^3}Q(x)|v_n|^6dx\to \tilde{l}_3,
$$
that is $\tilde{l}_3=\tilde{l}_1+\tilde{l}_2$ by \eqref{5.4g}.

Next we will show that $u_n\to u$ in $H^1(\R^3)$. Just suppose not, then $\tilde{l}_1>0$ and $\tilde{l}_3>0$.
Similar to \eqref{h11}, we also have
\begin{equation}\label{5.4h}
  \frac{l_1}{3}+\frac{l_2}{12}\geq \frac{abS^3}{4|Q|_{\infty}}+\frac{b^3S^6}{24|Q|^2_{\infty}}+\frac{(b^2S^4+4a|Q|_{\infty}S)^{\frac{3}{2}}}{24|Q|^2_{\infty}}.
\end{equation}
By \eqref{5.4c}-\eqref{5.4dddd}, we have
\begin{eqnarray*}
J(u_n)&=& J(u)+\frac{a+bA^2}{2}\int_{\R^3}|\nabla v_n|^2dx+\frac{1}{2}\int_{\R^3}|v_n|^2dx-\frac{1}{6}\int_{\R^3}Q(x)|v_n|^6dx+o(1)\\
&=&J(u)+\frac{1}{2}\|v_n\|^2+\frac{b}{4}\bigg(\int_{\R^3}|\nabla v_n|^2dx\bigg)^2+\frac{b}{4}\int_{\R^3}|\nabla v_n|^2dx\int_{\R^3}|\nabla u|^2dx\\
&&-\frac{1}{6}\int_{\R^3}|v_n|^6dx+\frac{bA^2}{4}\int_{\R^3}|\nabla v_n|^2dx+o(1)\\
&\stackrel{\mathrm{\eqref{5.4g}}}{=}&J(u)+\frac{1}{3}\|v_n\|^2+\frac{b}{12}\bigg[\bigg(\int_{\R^3}|\nabla v_n|^2dx\bigg)^2+\int_{\R^3}|\nabla v_n|^2dx\int_{\R^3}|\nabla u|^2dx\bigg]\\
&&+\frac{bA^2}{4}\int_{\R^3}|\nabla v_n|^2dx+o(1).
\end{eqnarray*}
Letting $n\to\infty$, and by \eqref{5.4d}
\begin{eqnarray*}
\tilde{c}+\frac{bA^4}{4}&=& J(u)+\frac{l_1}{3}+\frac{l_2}{12}+\frac{bA^2}{4}\int_{\R^3}|\nabla v_n|^2dx+o(1)\\
&\stackrel{\mathrm{\eqref{5.4h}}}{\geq}&J(u)+\frac{abS^3}{4|Q|_{\infty}}+\frac{b^3S^6}{24|Q|^2_{\infty}}+\frac{(b^2S^4+4a|Q|_{\infty}S)^{\frac{3}{2}}}{24|Q|^2_{\infty}}+\frac{bA^2}{4}\int_{\R^3}|\nabla v_n|^2dx+o(1)\\
&\stackrel{\mathrm{\eqref{5.4ff}}}{\geq} &-C_3\lambda^{\frac{2}{2-q}}+\frac{abS^3}{4|Q|_{\infty}}+\frac{b^3S^6}{24|Q|^2_{\infty}}+\frac{(b^2S^4+4a|Q|_{\infty}S)^{\frac{3}{2}}}{24|Q|^2_{\infty}}+\frac{bA^4}{4}
\end{eqnarray*}
which implies that
$$
\tilde{c}\geq \frac{abS^3}{4|Q|_{\infty}}+\frac{b^3S^6}{24|Q|^2_{\infty}}+\frac{(b^2S^4+4a|Q|_{\infty}S)^{\frac{3}{2}}}{24|Q|^2_{\infty}}-C_3\lambda^{\frac{2}{2-q}},
$$
a contradiction! Then $\tilde{l}_1=0$ which is $u_n\to u$ in $H^1(\R^3)$.
\end{proof}
\vskip0.3cm
\textbf{Proof of Theorem \ref{theorem3}:} Using Lemmas \ref{5.1}-\ref{5.3} and repeating the processes as in Sections 3.1-3.3, the proof of Theorem \ref{theorem3} is clear. \hspace{\stretch{1}}$\Box$

\vskip0.5cm

\noindent{\textbf{Acknowledgements}: Both authors would like to express  their sincere gratitude to Professor Gongbao Li for his interest and some insightful discussions about the Kirchhoff problem. The authors was supported by NSFC (Grant No. 11371158), the program for Changjiang Scholars and Innovative Research Team in University (No. IRT13066).}

\vskip0.5cm


\begin{thebibliography}{99}

\bibitem{ap}
G. Kirchhoff, Mechanik, Teubner, Leipzig, 1883.

\bibitem{apa}
M. Chipot, B. Lovat, Some remarks on nonlocal elliptic and parabolic problems, Nonlinear Anal. 30 (1997) 4619-4627.


\bibitem{aq}
A. Arosio, S. Panizzi, On the well-posedness of the Kirchhoff string, Trans. Amer. Math. Soc. 348 (1996) 305-330.



\bibitem{ar}
 S. Bernstein,  une classe d'\'{e}quations fonctionnelles aux d\'{e}riv\'{e}es partielles. Bulletin de I'Acad¨¦mie des Sciences de I'URSS. VII. S\'{e}rie 1940; 4:17-26.

\bibitem{as}
 P. D'Ancona, S. Spagnolo, Global solvability for the degenerate Kirchhoff equation with real analytic data, Invent. Math. 108 (1992) 247-262.

\bibitem{bc}
J. L. Lions, On some questions in boundary value problems of mathematical physics, in:
Contemporary Developments in Continuum Mechanics and Partial Differential Equations,
Proceedings of International Symposium, Inst. Mat., Univ. Fed. Rio de Janeiro, Rio de Janeiro,
1977, in: North-Holland Math. Stud. vol. 30, North-Holland, Amsterdam, 1978, pp. 284-346.


\bibitem{bl1}
M.M. Cavalcanti, V.N. Domingos Cavalcanti, J.A. Soriano, Global existence and uniform decay rates for the Kirchhoff-Carrier equation with nonlinear dissipation, Adv. Differential Equations 6 (2001) 701-730.

\bibitem{bl2}
G. M. Figueiredo, N. Ikoma, J. R. Santos Junior, Existence and concentration result for the Kirchhoff type equations with general nonlinearities, Arch. Rational Mech. Anal. 213 (2014) 931-979.

\bibitem{bj}
X. He, W. Zou, Infinitely many positive solutions for Kirchhoff-type problems, Nonlinear Anal. 70 (2009) 1407-1414.

\bibitem{bw}
Y. He, G. Li, S. Peng, Concentrating Bound States for Kirchhoff type problems in ${\mathbb{R}^3}$ involving critical Sobolev exponents, Adv. Nonlinear Stud. 14 (2014) 441-468.

\bibitem{c}
G. Li, H. Ye, Existence of positive ground state solutions for the nonlinear Kirchhoff type equations in ${\mathbb{R}^3}$, J. Differential Equations 257 (2014) 566-600.


\bibitem{c1}
W. Liu, X. He, Multiplicity of high energy solutions for superlinear Kirchhoff equations, J. Appl. Math. Comput. 39 (2012)  473-487.

\bibitem{c2}
K. Perera, Z. Zhang, Nontrivial solutions of Kirchhoff-type problems via the Yang index, J. Differential Equations 221 (2006) 246-255.


\bibitem{ckw}
A. Ambrosetti, H. Br\'{e}zis, G. Cerami, Combined effects of concave and convex nonlinearities in some elliptic problems, J. Funct. Anal. 122 (1994) 519-543.

\bibitem{cl1}
T. Barstch, M. Willem, On a elliptic equation with concave and convex nonlinearities, Proc. Amer. Math. Soc. 123 (1995)
3555-3561.
\bibitem{cl2}
T. S. Hsu, Multiple positive solutions for a critical quasilinear elliptic system with concave-convex nonlinearities, Nonlinear Anal. 71 (2009) 2688-2698.

\bibitem{cl3}
J. Garcia-Azorero, I. Peral, J. D. Rossi, A convex-concave problem with a nonlinear boundary condition, J. Differential
Equations 198 (2004) 91-128.
\bibitem{cl4}
M. Bouchekif, A. Matallah, Multiple positive solutions for elliptic equations involving a concave term and critical
Sobolev-Hardy exponent, Appl. Math. Lett. 22 (2009) 268-275.

\bibitem{cl5}
V. Benci and G. Cerami, Existence of positive solutions of the equation $ - \Delta u + a(x)u = {u^{\frac{{N + 2}}
{{N - 2}}}}$ in ${\mathbb{R}^N}$, J. Funct. Anal. 88 (1990) 90-117.


\bibitem{cl6}
A. Floer, A. Weinstein, Nonspreading wave packets for the cubic Schr\"{o}dinger equation with a bounded potential, J. Funct. Anal. 69 (1986) 397-408.


\bibitem{cl7}
Y. G. Oh, On positive multi-lump bound states of nonlinear Schr\"{o}dinger equations
under multiple well potential, Commun. Math. Phys. 131 (1990) 223-253.

\bibitem{cl8}
I. Ianni and G. Vaira, On concentration of positive bound states for the Schr\"{o}dinger-Poisson problem with potentials, Adv. Nonlinear Stud. 8 (2008) 573-595.

\bibitem{cl9}
J. Sun, T. F. Wu, Z. Feng, Multiplicity of positive solutions for a nonlinear Schr\"{o}dinger-Poisson system, J. Differential Equations 260 (2016) 586¨C627.



\bibitem{h}
C. Y. Chen, Y. C. Kuo, T. F. Wu, The Nehari manifold
for a Kirchhoff type problem involving sign-changing weight
functions, J. Differential Equations 250 (2011) 1876-1908.

\bibitem{hb}
S. H. Liang, S. Y. Shi, Soliton solutions to Kirchhoff type problems involving the critical growth in $\R^N$, Nonlinear Anal. 81 (2013) 31-41.

\bibitem{hc}
P. Pucci, M. Xiang, B. Zhang, Multiple solutions for nonhomogeneous Schr\"{o}dinger-Kirchhoff type equations involving the fractional $p$-Laplacian in $\R^N$, Calc. Var. Partial Differential Equations 54 (2015) 2785-2806.

\bibitem{hc1}
Y. He, G. Li, Standing waves for a class of Kirchhoff type problems in ${\mathbb{R}^3}$ involving critical Sobolev exponents, Calc. Var. Partial Differential Equations 54 (2015) 3067-3106.


\bibitem{i}
C. Y. Lei, G. S. Liu, L. T. Guo, Multiple positive solutions for a Kirchhoff type problem with a
critical nonlinearity, Nonlinear Anal. 31 (2016) 343-355.



\bibitem{jj}
F. Gazzola, M. Lazzarino, Existence results for general critical growth semilinear ellipitic equations,
Commun. Appl. Anal. 4 (2000) 39-50.

\bibitem{jj2}
L. Zhao, F. Zhao, Positive solutions for Schr\"{o}dinger-Poisson equations with a critical exponent, Nonlinear
Anal. 70 (2009) 2150-2164.

\bibitem{jj3}
Z. Liu, S. Guo, Existence and concentration of positive ground states for a Kirchhoff equation involving
critical Sobolev exponent, Z. Angew. Math. Phys. 66 (2015) 747-769.

\bibitem{jj4}
J. Wang, L. Tian, J. Xu, F. Zhang, Multiplicity and concentration of
positive solutions for a Kirchhoff type problem with critical
growth, J. Differential Equations  253 (2012) 2314-2351.

\bibitem{jj5}
G. Li, H. Ye, Existence of positive solutions for nonlinear
Kirchhoff type problems in $\R^3$ with critical
Sobolev exponent, Math. Meth. Appl. Sci. 37 (2014) 2570-2584.

\bibitem{jj5a}
W. Shuai, Sign-changing solutions for a class of Kirchhoff-type problem in bounded domains, J. Differential Equations 259 (2015) 1256-1274.
\bibitem{jj6}
X. Tang, B. Cheng, Ground state sign-changing solutions for Kirchhoff type problems in bounded domains, J. Differential Equations 261 (2016) 2384-2402.

\bibitem{jj7}
M. Willem, Minimax theorems, Progress in Nonlinear Differential Equations and their Applications, 24. Birkh\"{a}user Boston, Inc., Boston, MA, 1996.


\bibitem{ja}
H. Br\'{e}zis, L. Nirenberg, Positive solutions of nonlinear elliptic equations involving critical Sobolev exponents, Comm. Pure Appl. Math. 36 (1983) 437-477.


\bibitem{jaa}
P. L. Lions, The concentration-compactness principle in the calculus of variations, The locally compact case, part 2, Ann. Inst. H. Poincar\'{e} Anal. Non.  Lin\'{e}aire 2 (1984) 223-283.




\bibitem{jb}
P. L. Lions, The concentration-compactness principle in the calculus of variations, The limit case,
part 1, Rev. Mat. H. Iberoamericano 1.1 (1985) 145-201.

\bibitem{j}
I. Ekeland, Nonconvex minimization problems, Bull. Amer. Math. Soc. 1 (1979) 443-473.


\bibitem{jc}
G. Li, H. S. Zhou, The existence of a weak solution of inhomogeneous quasilinear
elliptic equation withcritical growth conditions, Acta Math. Sinica, 11 (1995) 146-155.

\bibitem{jd}
H. Br\'{e}zis, E.Lieb, Arelation between pointwise convergence of functions and convergence of functionals, Proc. Amer. Math. Soc. 88 (1983) 486-490.

\end{thebibliography}
\end{document}